\documentclass[twoside,14pt, a4paper]{extarticle}

\usepackage{amsmath}
\usepackage{amssymb}
\usepackage{amsbsy}
\usepackage[utf8]{inputenc}
\usepackage{dsfont}
\usepackage{fancyhdr}
\usepackage[colorlinks=true, allcolors=blue]{hyperref}
\usepackage{graphicx}

\usepackage{enumerate}

\def\a{{\boldsymbol a}}

\def\n{{\boldsymbol n}}
\def\u{{\boldsymbol u}}
\def\v{{\boldsymbol v}}
\def\x{{\boldsymbol x}}

\def\dx{{\rm d}{\boldsymbol x}}
\def\ds{{\rm d}s}
\def\dt{{\rm d}t}

\def\R {{\mathds R}}
\def\N {{\mathds N}}

\newtheorem{theorem}{Theorem}[section]
\newtheorem{lemma}[theorem]{Lemma}
\newtheorem{proposition}[theorem]{Proposition}
\newtheorem{corollary}[theorem]{Corollary}
\newtheorem{definition}[theorem]{Definition}

\newtheorem{remark}[theorem]{Remark}

\newenvironment{proof}{{\bf Proof.}}{\hfill $\square$ \newline}

\begin{document}

\title{Toward generalized solutions of  the Keller--Segel equations with singular sensitivity and signal absorption via an algebraic manipulation finite element algorithm}
\author{Juan Vicente Gutiérrez-Santacreu\thanks{Dpto. de Matemática Aplicada I, E. T. S. I. Informática, Universidad de Sevilla. Avda. Reina Mercedes, s/n. E-41012 Sevilla, Spain.  E-mail: {\tt juanvi@us.es}. JVGS acknowledges partial support from Grant PID2023-149182NB-I00 and from the IMUS-María de Maeztu grant CEX2024-001517-M, both funded by MICIU/AEI/10.13039/501100011033, with additional support from ERDF/EU.}}

\date{\today}

\maketitle



\begin{abstract} 

The paper that follows describes a numerical algorithm to solve the parabolic-parabolic Keller--Segel system characterized by singular sensitivity and signal absorption in such a manner that the numerical approximations converge towards a generalized solution on two-dimensional polygonal domains as the time and space discretization parameters tend to zero. 

The algorithm employs an algebraic manipulation finite element method for space, while time remains continuous, based on introducing a stabilized term. This term is constructed via a graph-Laplacian operator and a shock detector for detecting extrema in finite element functions. Furthermore, the cross-diffusion term  also includes an algebraic manipulation, which is related to testing by a nodally interpolated, suitable nonlinear function involved in obtaining a discrete energy-like law leading to a priori estimates. This approach yields approximations that respect physical constraints at nodal points such as positivity and maximum principle, and maintaining mass properties as well. Compactness results are quite laborious due to the low regularity stemming from the a priori estimates and the discretization procedure itself. Finally,  the passage to the limit rests on testing by the product of a positive test function and a renormalization of the numerical solution.

\end{abstract}

\tableofcontents
\section{Introduction}
\subsection{Keller--Segel approach}

\emph{Escherichia coli} (E. coli) is a species of bacteria commonly studied in microbiology and frequently used as a model organism due to its complex dynamics, including movement, aggregation, and spatial organization. These bacteria exhibit a movement pattern known as run-and-tumble motion, characterized by relatively straight movements interrupted by random tumbles to reorient themselves, which leads to a form of biased random walk. Aggregation occurs as a consequence of the presence of gradients of some cue detected through multiple transmembrane receptors. In response to these gradients, E. coli bacteria move towards regions of higher concentrations of attractants, such as nutrients. This biological response results in spatial organizations within the colony in their surrounding environment.      

In order to understand how E. coli bacteria propagate through the environment, Keller and Segel \cite{Keller_Segel_1971} proposed an initial framework for describing taxis-driven pattern formation and propagation fronts. In doing so, two mechanisms were regarded. One mechanism involved the consumption of the attractant upon interaction with bacteria, with the specific nature of this consumption dependent on the biological context. The other mechanism concerned the sensitivity of bacteria to the attractant being directly proportional to the reciprocal of the attractant density. This means that the sensitivity of the organisms rises proportionally as the chemoattractant density decreases and vice versa. This dynamic response mechanism obeys the fact that E. coli bacteria become more responsive to lower density concentrations of the attractant.

Let $\Omega\subset \mathds{R}^2$  be a bounded open subset and denote $\partial\Omega$ its boundary, which we assume to be Lipschitzian. The two-dimensional Keller--Segel equations \cite{Keller_Segel_1971} with logarithmic sensitivity and signal absorption  are written as 
\begin{equation}\label{KS_1971}
\left\{
\begin{array}{rclcc}
\displaystyle
\partial_t u-\Delta u+ \chi\nabla\cdot(\frac{u}{v} \nabla v)&=&0&\mbox{ in }&\Omega\times (0,\infty),
\\
\displaystyle
\partial_t v-\Delta v+ uv&=&0&\mbox{ in }&\Omega\times (0, \infty).
\end{array}
\right.
\end{equation}
Here $u: \overline{\Omega}\times [0,\infty)\to [0,\infty)$ represents the bacterial density and $v: \overline{\Omega}\times [0,\infty)\to (0,\infty) $ represents the attractant density, which needs to meet the boundary conditions
\begin{equation}\label{BC}
\partial_\n u=0 \quad \mbox{ and }\quad\partial_\n v =0 \quad \mbox { on } \quad  \partial \Omega\times (0,\infty),
\end{equation}
and the initial conditions
\begin{equation}\label{IC}
u(0)=u_0\quad
\mbox{ and }\quad v(0)=v_0\quad\mbox{ in }\quad \Omega,
\end{equation}
where $\partial_\n u$ or $\partial_\n v$ denotes the derivative of $u$ or $v$ in the direction normal $\n$ to $\partial\Omega$. 

The mathematical theory of problem \eqref{KS_1971}--\eqref{IC} is incomplete. Generalized solutions have been shown to exist for system \eqref{KS_1971}--\eqref{IC} in \cite{Winkler_2016}, but the global-in-time solvability of weak or classical solutions or the development of blow up phenomena remains still open. The only result derived thus far concerns the existence of classical solutions globally in time for initial data under a certain smallness condition \cite{Black_2018}. These generalized solutions stem from mild-regularity bounds and can be considered a set containing  weak solutions, since they must satisfy a variational inequality. This inequality rules out any finite-time collapse into persistent Dirac-type measures and is fully consistent with classical solutions. In a sense, these generalized solutions are reminiscent of the notion of suitable weak solutions for the Navier-Stokes equations \cite{Scheffer_1977}.  Another characteristic of the solutions to problem \eqref{KS_1971}--\eqref{IC} found in \cite{Winkler_2016} is the formation of planar traveling waves \cite{Wang_2013, Chae_Choi_Kang_Lee_2018}.

As of the time of writing this paper, there exists no known algorithm capable of effectively producing discrete solutions for problem \eqref{KS_1971}--\eqref{IC}. The main reason for such absence turns out to be that it is very tricky to find numerical solutions that meet the underline properties coming from the continuous analysis. Physical constraints such as positivity or maximum and minimum principles are key properties for the analysis of equations \eqref{KS_1971}. Therefore reproducing these properties at the discrete level is a feature that is highly desired in applications and greatly facilitates the convergence analysis of algorithms from a numerical perspective. A priori bounds for problem \eqref{KS_1971}--\eqref{IC}, which yield generalized solutions, are typically based on energy-like laws rooted in physical constraints. These bounds are obtained by testing against nonlinear functions. However, this approach presents some difficulties, particularly when employing finite element methods, since nonlinear functions do not belong to the finite element space being linear. This variance leads to an outstanding obstacle in devising effective numerical algorithms for problem \eqref{KS_1971}--\eqref{IC}.                   
 
If the logarithmic sensitivity is substituted with a linear sensitivity and the consumption term is replaced by a linear production and a degradation of attractant, one is led to the Keller--Segel equations \cite{Keller_Segel_1970}  
\begin{equation}\label{KS_1970}
\left\{
\begin{array}{rcll}
\partial_tu-\Delta u&=&-\nabla\cdot(u\nabla v)&\mbox{ in } \Omega\times(0,\infty),
\\
\partial_t v -\Delta v&=&u-v&\mbox{ in } \Omega\times(0,\infty).
\end{array}
\right.
\end{equation}
This system has been subject to an intensive numerical investigation in recent years and is closely related to  \eqref{KS_1971}. Both systems \eqref{KS_1971} and \eqref{KS_1970} satisfy physical constraints and a priori bounds coming from testing by nonlinear functions. One important difference between both systems is the concept of solution. It is known \cite{Nagai_Senba_Kiyoshi_1997} that if the total mass of bacterial density is bounded by $4\pi$ , a classical solution exists for system \eqref{KS_1970}; otherwise, there are initial data that lead to blowup phenomena \cite{Herrero_Velazquez_1997, Horstamann_Wang_2001}. Despite the development of several numerical strategies for tackling system \eqref{KS_1970}, very few of them consider all of the above-mentioned properties  \cite{Huang_Shen_2021, Badia_Bonilla_GS_2023, AS_GG_RG_2023}. 

The convergence analysis for systems \eqref{KS_1971} and \eqref{KS_1970} is not typically addressed in the literature.  Passing to the limit of the would-be approximate solutions of system \eqref{KS_1971}  is far more involved than that for system \eqref{KS_1970}. This is due to the fact that \emph{a priori} bounds for system \eqref{KS_1970}  provide enough control over derivatives for the approximate solutions and hence the concept of solution does differ.      

In this paper, we propose an algebraic manipulation finite element method, which enhances the standard finite element method for \eqref{KS_1971} by incorporating a stabilizing term. This term utilizes a shock detector in combination with a graph-Laplacian operator to achieve stabilization, which avoids the appearance of non-physical oscillations. Numerical solutions from the resulting method enjoy physical constraints such as positivity and  a discrete maximum principle; furthermore mass is preserved. In obtaining a priori bounds, one needs to test $\eqref{KS_1971}_1$ against the nonlinear function $\displaystyle\frac{1}{1+u}$ to get bounds of $\log(1+u)$.  The numerical treatment of the cross-diffusion term $-\nabla\cdot(u\nabla v)$ in $\eqref{KS_1971}_1$  is the major source of difficulties.  An approximation of $u$ is constructed by using $\frac{1}{1+u}$ and $\log(1+u)$ for algebraically manipulating $-\nabla\cdot(u\nabla v)$ . This manipulation is compatible with attaining a priori bounds over $\log(1+u)$ at the discrete level. Additionally,  mass lumping is applied to the time and reaction terms. 

The convergence analysis rests on Vitali's convergence theorem for proving that  the approximations of $u$ converge strongly in $L^1_{\rm loc}(0,\infty; L^1(\Omega))$ and of $-\log(\frac{v}{\|v_0\|_{L^\infty(\Omega)}})$ in \break $L^2_{\rm loc}(0,\infty; H^1(\Omega))$. This last convergence is accomplished by comparing $\eqref{KS_1971}_2$ with its numerical approximation. The passage to the limit requires using Chebyshev's inequality and the continuity of the Lebesgue integral.

\subsection{Outline}

Section 2 begins with some notation and some assumptions about the domain and the triangulation thereof.  Then the numerical tools for approximating are presented to develop the finite element method. Finally, the statement of the main theorem concerning the convergence of the numerical solutions to a generalized solution of problem \eqref{KS_1971}--\eqref{IC} is stated. Before demonstrating our main result, Section 3 compiles various technical preliminaries. In Section~4, the proof is organized in several subsections: positivity and maximum discrete principle, a priori bound, weak and strong convergences, and passage to the limit. 

\section{Main result}
This section is devoted to setting out  the assumptions and tools utilized for developing the finite element approximation of problem \eqref{KS_1971}--\eqref{IC}. It ends up with the presentation of our main theorem regarding convergence toward generalized solutions, including their definition.

\subsection{Assumptions and notation} 
With a view toward starting our main theorem we enumerate the required assumptions. These involve the domain, the family of meshes, and the associated finite element spaces. 
\begin{enumerate}
\item [(A1)] Let $\Omega$ be a bounded domain of $\R^2$ with a polygonal boundary.
\item[(A2)] Let $\{{\mathcal T}_{h}\}_{h>0}$  be a family of weakly acute, quasi-uniform triangulations of  $\overline{\Omega}$ made up of triangles, so that $\overline \Omega=\cup_{T\in {\mathcal T}_h}T$, where $h=\max_{T\in \mathcal{T}_h} h_T$, with $h_T$ being the diameter of $T$. Moreover,  let  $\mathcal{N}_h=\{\a_i \}_{i=0}^I$ be the coordinates of the nodes of $\mathcal{T}_h$.

\item [(A3)]  Associated with  ${\mathcal T}_h$ is the finite element space
$$
X_h = \left\{ x_h \in {C}^0(\overline\Omega) \;:\;
x_h|_T \in \mathcal{P}_1(T), \  \forall T \in \mathcal{T}_h \right\},
$$
where $\mathcal{P}_1(T)$ is the set of linear polynomials on  $T$.

\end{enumerate}

The dual of $X_h$ is denoted as $X'_h$ and the set of linear operators from $X'_h$ to $X_h$ as $\mathcal{L}(X_h', X_h)$. Let $\{\varphi_{\a_i}\}_{i\in I}$ be the standard basis functions for $X_h$, where $\Delta_{\a_i}={\rm supp}\, \varphi_{\a_j}$ and $\mathcal{N}_h(\Delta_{\a_i})=\{\a_j\in \mathcal{N}_h \,:\, \a_j\in \Delta_{\a_i}\}$. For $i\in I$, define $\mathcal{N}_h^{\rm sym}(\Delta_{\a_i})$ as the set of symmetric nodes $\boldsymbol{a}_{ij}^{\rm sym}$ with respect to $\boldsymbol{a}_i\in\mathcal{N}_h$. These symmetric nodes are constructed as points at the intersection between the line passing through $\a_i$ and $\a_j$ and $\partial\Omega{\a_i}$, excluding $\a_j$. See Figure~\ref{fig:a_ij^sym} for illustration.
\begin{figure}[h!]
\centering
\includegraphics[width=0.3\textwidth]{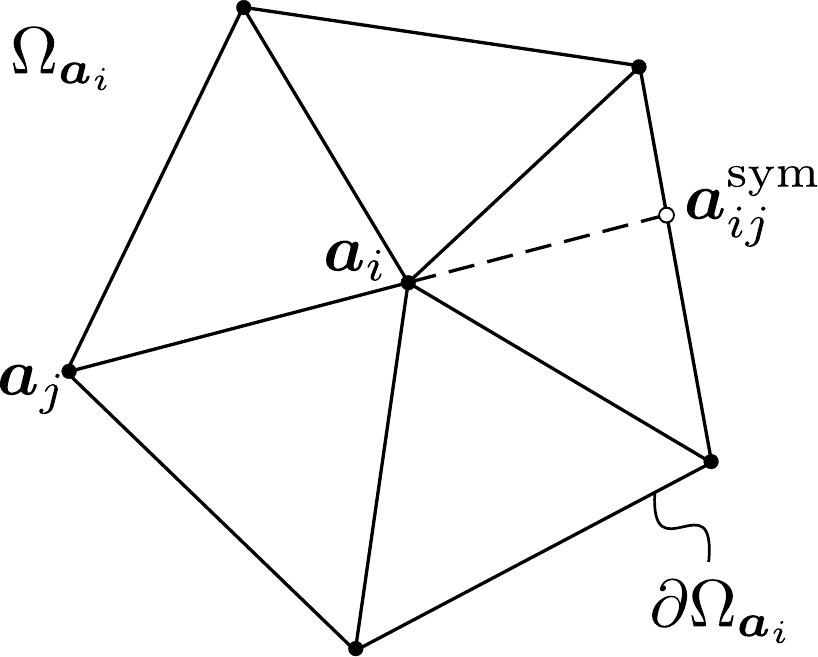}
\caption{Symmetric node $\a^{\rm sym}_{ij}$ to $\a_j$ from $\a_i$}
\label{fig:a_ij^sym}
\end{figure}

We denote $i_h : C(\Omega) \to  X_h$ as the standard nodal interpolation operator, such that $i_h x(\a_i)=x(\a_i)$ for all $i\in I$. The discrete inner product  $(\cdot,\cdot)_h : X_h\times X_h\to \R$ on $X_h$ is defined by
$$
(x_h,\bar x_h)_h=\int_\Omega i_h(x_h(\x)\bar x_h(\x))\, \dx,
$$
associated with 
$$
\|x_h\|_h=(x_h,x_h)_h^{\frac{1}{2}}.
$$

For each $i\in I$, we associate $T_{\a_i}\in \mathcal{T}_h$ such that $\a_i\in T_{\a_i}$. Then, if $x\in L^1(\Omega)$, define $j_h: L^1(\Omega)\to X_h$ to be the average interpolation operator as 
\begin{equation}\label{def:jh}
j_h x(\a)=\frac{1}{|T_{\a_i}|}\int_{T_{\a_i}} x(\x) \dx.
\end{equation}
and $sz_h: L^1(\Omega)\to X_h$ to be the Scott-Zhang interpolation operator constructed in \cite{Girault_Lions_2001, Scott_Zhang_1990}.
   
Let $q_h: L^1(\Omega) \to X_h $ be such that  
\begin{equation}\label{def:qh}
(q_h x, x_h)_h=(x, x_h)\quad \mbox{for all}\quad x_h\in X_h.
\end{equation}

Finally, we introduce some shorthand notation. Given $x_h\in X_h$, we  write $x_h(\a_i)=x_i$ for all $i\in I$ and $\delta_{ij}x_h=x_i-x_j$.

\begin{remark} It is well here to remark that the condition on $\mathcal{T}_h$ of being weakly acute refers to the property that, for all $\varphi_{\a_i}, \varphi_{\a_j}$, where $i\neq j$,
$$
(\nabla \varphi_{\a_j}, \nabla\varphi_{\a_i})\le 0
$$
holds. This is guaranteed when  the sum of opposite angles relative to any side does not exceed $\pi$, which is ensured by the Delaunay triangulations.  
\end{remark}

\subsection{Finite element approximation }
We restrict ourselves to a semidiscrete, continuous-in-time, finite element approximation for problem \eqref{KS_1971}--\eqref{IC}.  The staring point is to manipulate algebraically a standard finite element approximation for \eqref{KS_1971} by introducing two stabilizing terms --one for each finite element equation. These stabilizing terms are thought to satisfy different physical bounds: ensuring positivity in the numerical solutions of $\eqref{KS_1971}_1$ and enforcing positivity and a discrete maximum principle for the numerical solutions of $\eqref{KS_1971}_{2}$; thereby guaranteeing that the subsequent analysis is well-defined. Furthermore, the new discretization of the cross-diffusion term, built to tackle an energy-like inequality, must be compatible with preserving the above-mentioned physical bounds and satisfy a discrete Leibniz rule for spatial derivatives of nonlinear functions; a key aspect of generalized solutions.

A previous step to our purposes is to rewrite the cross-diffusion term as 
$$
\nabla\cdot(\frac{u}{v}\nabla v)=\nabla\cdot(u\nabla\log v).
$$

Let $(u_{0h}, v_{0h})\in X_h\times X_h$ be an approximation of $(u_0, v_0)$. Then our numerical algorithm reads as.  Set $u_h^0=u_{0h}$, $v^0_h=v_{0h}$ and find  $(u_h,  v_h)\in C^1([0,\infty); X_h)\times C^1([0,\infty); X_h)$ such that, for all $(\bar u_h,\bar v_h)\in X_h\times X_h$,  
\begin{equation}\label{eq:u_h}
\begin{array}{c}
\displaystyle
(\partial_t u_h, \bar u_h)_h+(\nabla u_h, \nabla \bar u_h)
-(u_h\nabla i_h \log v_h, \nabla \bar u_h)_*+(B(u_h, v_h) u_h, \bar u_h)=0,
\end{array}
\end{equation}
\begin{equation}\label{eq:v_h}
\begin{array}{rcl}
(\partial_t v_h, \bar v_h)_h +(\nabla v_h,\nabla \bar v_h)+(u_h v_h, \bar v_h)_h=0.
\end{array}
\end{equation}

The term  $(u_h\nabla i_h \log v_h, \nabla \bar u_h)_*$ is designed to enable the derivation of \emph{a priori} estimates through testing with nonlinear functions (e.g., $\bar u_h = -i_h\frac{1}{1+u_h}$), which is essential in handling the singular sensitivity structure.

The stabilization term $(B(u_h, v_h) u_h, \bar u_h)$ is constructed consistently with this discretization to preserve the nonnegativity of $u_h$. Indeed, relying solely on the weak acuteness of the mesh is not sufficient to guarantee this property at the discrete level.

Therefore, this term compensates for the loss of monotonicity induced by the nonlinear chemotaxis contribution and ensures discrete nonnegativity.

To the best of our knowledge, this specific combination of discretization and stabilization for singular sensitivity has not been previously proposed.

The new discrete cross-diffusion term is designed in the spirit of \cite{Badia_Bonilla_GS_2023,Bonilla_GS_2024}. It begins with  
$$
\begin{array}{rcl}
(x_h\nabla i_h \log \tilde x_h,\nabla \bar x_h)&=&\displaystyle
\sum_{k,j,i\in I} x_k \log \tilde x_j  \bar x_i (\varphi_{\a_k} \nabla \varphi_{\a_j}, \nabla\varphi_{\a_i}) 
\\
&=&\displaystyle
\sum_{\mbox{\tiny $\begin{array}{c}k\in I \\i<j\in I\end{array}$}} x_k (\log \tilde x_j- \log \tilde x_i)  (\bar x_i-\bar x_j) (\varphi_{\a_k} \nabla \varphi_{\a_j}, \nabla\varphi_{\a_i}).
\end{array}
$$
and approximates $x_k$ by $\tau_{ij}$, depending on $x_i$ and $x_j$. This leads to 
\begin{equation}\label{New_KS_discretization}
(x_h\nabla i_h \log\tilde x_h, \nabla \bar x_h)_*=\sum_{i<j\in I}\tau_{ji}(x_h) \delta_{ji}\log\tilde x_h \delta_{ij}\bar x_h(\nabla \varphi_{\boldsymbol{a}_j},\nabla\varphi_{\boldsymbol{a}_i}),
\end{equation}
with
\begin{equation}\label{def:tau_ij}
\tau_{ji}(x_h)=\left\{
\begin{array}{ccl}
\displaystyle
\Lambda_{ij}(u_h) \frac{\delta_{ji}\log(1+x_h)}{\delta_{ij}\frac{1}{1+x_h}}&\mbox{ if }& x_j\not=x_i,
\\
 x_i&\mbox{ if }& x_j=x_i,
\end{array}
\right.
\end{equation}
where
$$
\Lambda_{ij}(u_h)=\frac{1}{2}\left(\frac{x_i}{1+x_i}+\frac{x_j}{1+x_j}\right).
$$

Before defining the stabilizing term $B$, we need to introduce the shock detectors $\alpha_{\a_i}$   associated with each one. 
For each $i \in I$, let $\alpha_{\a_i}: X_h\to \R$  be such that, for each $q \in \mathds{R}^+$ and $x_h\in X_h$, 
$$
\alpha_{\a_i}(x_h) = \left\{
\begin{array}{cc}  
\left[
\frac{\left[{\sum_{j\in I(\Delta_{\a_i})} [\![\nabla x_h]\!]_{ij}}\right]_+}{\sum_{j\in I(\Delta_{\a_i})} 2\{\!\!\{|\nabla x_h \cdot \hat{\boldsymbol{r}}_{ij}|\}\!\!\}_{ij}}
\right]^q & \text{if } 
\sum_{j\in I(\Delta_{\a_i})} \{\!\!\{|\nabla x_h\cdot \hat{\boldsymbol{r}}_{ij}|\}\!\!\}_{ij} \neq 0, 
\\
0 & \text{otherwise},
\end{array}
\right. 
$$
where $[a]_{+}=\max\{0,a\}$, and 
$$
[\![\nabla x_h]\!]_{ij} =\frac{x_j - x_i}{|\boldsymbol{r}_{ij}|} + \frac{x_j^{\rm sym} - x_i}{|\boldsymbol{r}_{ij}^{\rm sym}|}, 
$$
and
$$
\{\!\!\{|\nabla x_h\cdot \hat{\boldsymbol{r}}_{ij}|\}\!\!\}_{ij} =\frac{1}{2}
\left(\frac{|x_j-x_i|}{|\boldsymbol{r}_{ij}|}+\frac{|x_j^{\rm sym}-x_i|}{|\boldsymbol{r}_{ij}^{\rm sym}|}\right),
$$
with $\boldsymbol{r}_{ij} = \a_j - \a_i$, whose normalization is $\hat{\boldsymbol{r}}_{ij} = \frac{\boldsymbol{r}_{ij}}{|\boldsymbol{r}_{ij}|}$, and $\boldsymbol{r}_{ij}^{\rm sym} = \a_{ij}^{\rm sym} - \a_i$ for $\a^{\rm sym}_{ij}\in\mathcal{N}^{\rm sym}_h(\Delta_{\a_i})$. 

The key characteristic of $\alpha_\a$, for $a\in\mathcal{N}_h$, lies in their role in localizing extrema. This is formalized in the following.   
\begin{lemma}\label{lm:alpha_i} If $x_h\in X_h$ reaches a minimum at $\a\in\mathcal{N}_h$, then it follows that $\alpha_{\a}(x_h)=1$; otherwise $ 0\le \alpha_{\a}(x_h)\le1$ for all $\a\in\mathcal{N}_h$. 
\end{lemma}

In defining the stabilizing operator $B$, we proceed as follows. For each $x_h, \bar x_h\in X_h\times X_h$, let $B(x_h, \bar x_h): X_h\times X_h\to \mathcal{L}(X'_h,X_h)$ be such that  
\begin{equation}\label{Bu}
(B(x_h, \bar x_h) \tilde x_h,  \hat x_h)=\sum_{i<j\in I}\beta_{ji}(x_h, \bar x_h) \delta_{ji} \tilde x_h \delta_{ji}\hat x _h\quad \mbox{ for all }\quad  \tilde x_h,\hat x_h\in X_h,
\end{equation}
where
\begin{equation}\label{def:beta^u_ji}
\beta_{ji}(x_h,\bar x_h)=\left\{
\begin{array}{ccl}
\max\{\alpha_{\a_i}(x_h) f_{ij}(x_h, \bar x_h), \alpha_{\a_j}(x_h)f_{ji}(x_h, \bar x_h), 0\} &\mbox{ for } & i\not= j,
\\[1.ex]
\displaystyle
\sum_{j\in I(\Delta_{\a_i})\backslash\{i\}}\beta_{ij}(x_h, \bar x_h) &\mbox{ for } & i= j,
\end{array}
\right.
\end{equation}
with
\begin{equation}\label{def:fij^u}
f_{ij}(x_h, \bar x_h)=\left\{
\begin{array}{ccl}
\displaystyle
\delta_{ji}\log\bar x_h \Big[\frac{x_i}{\delta_{ji} x_h}-\frac{\tau_{ji}(x_h)}{\delta_{ji}x_h}\Big] (\nabla \varphi_{\a_j}, \nabla \varphi_{\a_i})&\mbox{ if }& x_j\not=x_i,
\\
0&\mbox{ if }&x_j=x_i.
\end{array}
\right.
\end{equation}

\subsection{Statement of the result }
In the context of the equations modeling chemotaxis, it is sometimes unavoidable to use a suitable concept of solution compatible with classical solutions. For certain Keller--Segel models it is extremely difficult to find smooth solutions, in the sense that all the space-time derivatives in the model exist in the usual classical sense, or even weak solutions, where space-time derivatives are weakly defined in Sobolev spaces. The new framework is the concept of generalized solutions, which can be understood as a weak supersolution and is closely connected to smooth solutions. This connection makes reference to the fact that generalized solutions that are smooth satisfy system \eqref{KS_1971} classically. With this more general definition of solution it is possible to prove existence (but not uniqueness).

The concept of generalized solution for problem \eqref{KS_1971}-\eqref{BC} is understood as follows.
\begin{definition}\label{def:gener_sol} Assume that $(u_0, v_0)\in L^1(\Omega)\times L^\infty(\Omega)$ are such that $u_0\ge 0$ and $v_0>0$. A pair $(u,v)$ is called a generalized solution of the boundary-value problem \eqref{KS_1971}--\eqref{IC} if
$$
u\ge 0\mbox{ a. e. in } \Omega\times(0,\infty),  
$$
$$
0<v<\|v_0\|_{L^\infty(\Omega)}  \mbox{ a. e. in } \Omega\times(0,\infty),
$$
$$
u\in L^1_{\rm loc}(0,\infty; L^1(\Omega))
$$ 
and 
$$
v\in L^\infty((0,\infty)\times\Omega)\cap L^2_{\rm loc}(0,\infty; H^1(\Omega))
$$
in addition to 
$$
\nabla\log(u+1)\in L^2_{\rm loc}(0,\infty; L^2(\Omega))
$$
and
$$
\nabla \log v\in L^2_{\rm loc}(0,\infty; L^2(\Omega))
$$
such that
$$
\begin{array}{c}
\displaystyle
-\int_0^\infty (\log(u + 1),\varphi)\, \dt - (\log(u_0 + 1), \varphi(0)) \ge \int_0^\infty(|\nabla \log(u + 1)|^2,\varphi)\,\dt
\\
\displaystyle
- \int_0^\infty (\nabla \log(u + 1), \nabla\varphi) \dt -\int_0^\infty(\frac{u}{u + 1} \nabla\log(u + 1), \varphi\nabla \log v ) \dt 
\\
\displaystyle
+ \int_0^\infty(\frac{u}{u + 1} \nabla\log v, \nabla\varphi)\,\dt
\end{array}
$$
for all $\varphi\in C_0^\infty (\Omega \times [0, \infty ))$ with $\varphi\ge0$, and
$$
\int_0^\infty(v, \partial_t\psi)\dt - (v_0, \psi(0))=\int_0^\infty(\nabla v, \nabla\psi)\dt+\int_0^\infty (vu, \psi) \,\dt
$$
for all $\psi\in C_0^\infty (\Omega \times [0, \infty ))$.
\end{definition}

We are now prepared to present the main theorem of this paper.
\begin{theorem}\label{Th:main} Assume that assumptions $\rm (A1)$--$\rm(A3)$ are satisfied. Let $(u_0,v_0)\in L^1(\Omega)\times L^\infty(\Omega)$ with $u_0\ge 0$ and $ v_0 > 0$ such that  there exists $\vartheta>0$ satisfying  
\begin{equation}\label{cond_log:v_0}
-\int_\Omega \log\frac{v_0(\x)}{\|v_0\|_{L^\infty(\Omega)}}\le\vartheta,
\end{equation}
and take $(u_{0h}, v_{0h})\in X_h\times X_h$ such that  $u_{0h}=j_h u_0$ and $v_{0h}=j_h v_0$.  Then the discrete solution pair $\{(u_h, v_h)\}_{h>0}$ constructed by algorithm \eqref{eq:u_h} and \eqref{eq:v_h} converges, up to subsequences, toward a generalized solution of problem \eqref{KS_1971}-\eqref{IC} on $(0,\infty)$. 
\end{theorem}

\begin{remark} A more detailed analysis of our algorithm when discretized using an Euler time-stepping scheme can certainly be developed, addressing properties such as positivity, nonnegativity, the discrete maximum principle, \emph{a priori} bounds, and the passage to the limit. However, including a time discretization would lead to an unnecessarily lengthy exposition, as the main analytical challenges arise from the spatial discretization, as will become clear. 
\end{remark}
\section{Technical preliminaries}
For later purposes it will be useful to collect some properties of the space $X_h$ constructed on $\mathcal{T}_h$ that  need not to be weakly acute.  
\begin{proposition}
Let $0 \le m \le \ell\le 1$ and $1\le p, q \le\infty$. Then, there is $C_{\rm inv}>0$, independent of $h$, $T$, $p$, and $q$, such that, for $x_h\in X_h$
\begin{equation}\label{inv_local:WlpToWmq}
\|x_h\|_{W^{\ell,p}(T)}\le C_{\rm inv} h^{m-\ell+2(\frac{1}{p}-\frac{1}{q})} \|x_h\|_{W^{m,q}(T)}
\end{equation}
and
\begin{equation}\label{inv_global: WlpToWmq}
\|x_h\|_{W^{\ell,p}(\Omega)}\le C_{\rm inv} h^{m-\ell+2\min\{0,\frac{1}{p}-\frac{1}{q})\}} \|x_h\|_{W^{m,q}(\Omega)}.
\end{equation}
\end{proposition}
\begin{proof}
See \cite[Lm. 4.5.3]{Brenner_Scott_2008} or \cite[Lm. 1.138]{Ern_Guermond_2004} for a proof.
\end{proof}

The following propositions concern some error bounds, and stability and commutator  properties for the interpolation operators $i_h$, $j_h$ and $q_h$ as defined above.  

\begin{proposition}
There exist three positive constants $C_{\rm err}$, $C_{\rm com}$ and $C_{\rm sta}$, independent of $h$, such that \begin{equation}\label{err_LinfW1inf:i_h}
\|x-i_h x\|_{L^\infty(\Omega)}\le C_{\rm err } h \|x\|_{W^{1,\infty}(\Omega)},
\end{equation}
\begin{equation}\label{err_L1W21:i_h}
\|x-i_h x\|_{L^1(\Omega)}\le C_{\rm err } h^2 \|x\|_{W^{2,1}(\Omega)}, 
\end{equation}
\begin{equation}\label{err_H1H2:i_h}
\|x-i_h x\|_{H^1(\Omega)}\le C_{\rm err } h \|x\|_{H^2(\Omega)},
\end{equation}
\begin{equation}\label{Com_err_L2H1:i_h}
\|x_h \overline x_h-i_h(x_h  \bar x_h)\|_{L^1(\Omega)}\le C_{\rm com} h\, \|x_h\|_{L^2(\Omega)} \, \|\nabla \overline x_h\|_{L^2(\Omega)},
\end{equation}
\begin{equation}\label{Com_err_H1'W1inf:i_h}
\|x_h\bar x_h-i_h(x_h\bar x_h)\|_{L^1(\Omega)}\le C_\mathrm{com} h^{\frac{1}{2}} \|x_h\|_{(H^1(\Omega)\cap L^\infty(\Omega))'} \|\nabla \bar x_h \|_{L^\infty(\Omega)},
\end{equation}
\begin{equation}\label{Stab:i_h}
\|x_h^n\|_{L^1(\Omega)}\le \|i_h(x^n_h)\|_{L^1(\Omega)}\le C_{\rm sta} \|x_h^n\|_{L^1(\Omega)}\quad\mbox{ for all }\quad n\in\mathds{N}\quad\mbox{ with }\quad x_h\ge0,
\end{equation}
and
\begin{equation}\label{Stab_Linf:ih}
\|i_h x\|_{W^{k,\infty}(\Omega)}\le \|x\|_{W^{k,\infty}(\Omega)}\quad\mbox{ for }\quad k=0,1.
\end{equation}
\end{proposition}
\begin{proof}
The interpolation errors \eqref{err_LinfW1inf:i_h} and \eqref{err_H1H2:i_h} can be found in \cite[Th. 4.4.20]{Brenner_Scott_2008} or \cite[Co.  1.109]{Ern_Guermond_2004}. The commutator properties \eqref{Com_err_L2H1:i_h} and \eqref{Com_err_H1'W1inf:i_h}  are proved in \cite{GG_GS_2019} and \cite{C_GS_RG_2020}, respectively. Finally, the stability property \eqref{Stab:i_h} is in  \cite{GS_RG_2021},  and \eqref{Stab_Linf:ih} is proved by using that the space $X_h$ is made of piecewise linear elements.  
\end{proof}

\begin{proposition} There exist two constants $C_{\rm err}>0$ and $C_{\rm sta}>0$, independent of $h$, such that, for  $p=1$ and  $\infty$,  
\begin{equation}\label{err_LpW1p:jh}
\| x- j_h x \|_{L^p (\Omega)}\le C_{\rm err}  h \| x \|_{W^{s(p),p}(\Omega)},\quad\mbox{ with }\quad s(1)=2\mbox{ and } s(\infty)=1.
\end{equation}
and
\begin{equation}\label{stab_Lp:jh}
\|j_h x\|_{L^p(\Omega)}\le C_{\rm sta} \| x \|_{L^p(\Omega)}.
\end{equation}
\end{proposition}
\begin{proof} In proving  \eqref{err_LpW1p:jh} for $p=\infty$, we have, by \eqref{err_LinfW1inf:i_h},  that 
$$
\begin{array}{rcl}
\|x- j_h x \|_{L^\infty(\Omega)}&\le& \|x- i_h x\|_{L^\infty(\Omega)}+\|i _h x - j_h x \|_{L^\infty(\Omega)}
\\
&\le& C_{\rm err} h \|x \|_{W^{1,\infty}(\Omega)}+\|i _h x- j_h x\|_{L^\infty(\Omega)}. 
\end{array}
$$
Thus we obtain, for certain $\a\in\mathcal{N}_h$, that  
$$
\begin{array}{rcl}
\|i _h x- j_h x\|_{L^\infty(\Omega)}&=&\displaystyle
| i_h x(\a)-j_h x(\a)|= \left|x(\a)-\frac{1}{T_\a} \int_{T_\a} x(\x)\,\dx\right|
\\
&\le&\displaystyle
\frac{1}{T_\a} \int_{T_\a}|x(\a)-x(\x)|\,\dx\le h \|\nabla x\|_{L^\infty(T_\a)}
\\
&\le & h \|\nabla x\|_{L^\infty(\Omega)}.
\end{array}
$$
For $p=1$, the procedure is very similar, but using \eqref{err_L1W21:i_h} and 
$$
\|i_h x- j_h x\|_{L^1(\Omega)}=\sum_{\a\in\mathcal{N}_h} \left|x(\a)-\frac{1}{T_\a} \int_{T_\a} x(\x)\,\dx\right| \int_{\Delta_{\a}}\varphi_{\a}(\x)\dx
$$
together with the uniformity of  $\mathcal{K}_h$.
 
Inequality  \eqref{stab_Lp:jh} follows from the very definition. 
\end{proof}

\begin{proposition} There exist two constants $C_{\rm err}>0$ and $C_{\rm sta}>0$, independent of $h$, such that  
\begin{equation}\label{Error_L2-H1:qh}
\|x- q_h x\|_{L^2(\Omega)}+h \|\nabla( x- q_h x)\|_{L^2(\Omega)}\le C_{\rm err} h \|\nabla x\|_{L^2(\Omega)},
\end{equation}
\begin{equation}\label{Stab_Linf:qh}
\|q_h x\|_{L^\infty(\Omega)}\le C_{\rm sta} \|x\|_{L^\infty(\Omega)},
\end{equation}
\begin{equation}\label{Stab_H1:qh}
\|q_h x\|_{H^1(\Omega)}\le C_{\rm sta} \|x\|_{H^1(\Omega)}.
\end{equation}
\end{proposition}
\begin{proof} Assertion \eqref{Error_L2-H1:qh} is demonstrated in  \cite{Barrett_Blowey_1999, Barrett_Blowey_Garcke_1999}.

Let us take $x_h=i_h q_h^{p-1} x$ in \eqref{def:qh}, with $p$ even,  and use \eqref{Stab:i_h} to get
$$
\begin{array}{rcl}
\|i_h q^p_h x\|_{L^1(\Omega)}&\le& \|x\|_{L^\infty(\Omega)} \|i_h q_h^{p-1} x\|_{L^1(\Omega)}\le C \|x\|_{L^\infty(\Omega)} \|q^{p-1}_h x\|_{L^1(\Omega)}
\\
& \le & C \|x\|_{L^\infty(\Omega)} |\Omega|^\frac{1}{p} \|q^{p}_h x\|^{\frac{p-1}{p}}_{L^1(\Omega)}
\end{array}
$$
and hence
$$
\|q^p_h x\|_{L^p(\Omega)}\le C \|x\|_{L^\infty(\Omega)}  |\Omega|^\frac{1}{p}.
$$
Taking $p\to\infty$ leads to \eqref{Stab_Linf:qh}. Statement \eqref{Stab_H1:qh} follows straightforwardly from \eqref{Error_L2-H1:qh}.
\end{proof}

For a proof of the following proposition, see \cite{Scott_Zhang_1990, Girault_Lions_2001}. 
\begin{proposition} There exists an (average) interpolation operator  $sz_h:L^1(\Omega) \to X_h$ such that
\begin{equation}\label{sta:sz_h}
\|sz_h x\|_{W^{s,p}(\Omega)}\le C_{\rm sta} \| x \|_{W^{s,p}(\Omega)}\quad \mbox{for }  s=0,1\mbox{ and } 1\le p\le\infty,
\end{equation}
and
\begin{equation}\label{error:sz_h}
\|sz_h x- x \|_{W^{s,p} (\Omega)}\le C_{\rm app}  h^{m-s} \|x\|_{W^{m,\infty}(\Omega)}\quad \mbox{for } 0\le s\le m\le 2.
\end{equation}
\end{proposition}

The following propositions are borrowed from \cite{Barrenechea_Burman_Karakatsani_2017} and \cite{Bonilla_GS_2024}, respectively.
\begin{proposition} Let $\{a_n\}_{n=0}^N$ and $\{b_n\}_{n=0}^N$ two sequence of real numbers. Then  
\begin{equation}\label{ineq:sum}
\Big|\dfrac{\sum_{n=0}^N a_n}{\sum_{n=0}^N|a_n|}- \dfrac{\sum_{n=0}^N b_n}{\sum_{n=0}^N|b_n|}\Big|\le 2 \frac{\sum_{n=0}^N | a_n-b_n|}{\sum_{n=0}^N|a_n|}.
\end{equation}
\end{proposition}

\begin{proposition} It follows that, for all $x, y>0$, 
\begin{equation}\label{ineq:log}
(\log(x)-\log(y))^2\le \frac{(x-y)^2}{xy}
\end{equation}
and
\begin{equation}\label{ineq:log_II}
|\frac{y}{1+y}-\frac{x}{1+x}|\le |\log(1+x)-\log(1+y)|
\end{equation}
hold.
\end{proposition}
\begin{proof} Let $g(z)= (z-1)^2-z \log^2 z$.  We know that $g(z)\ge 0$ for $z>0$ since $g$ attains its unique global minimum $0$ at $z=1$. Therefore, we deduce that  $\log^2 z\le \frac{(z-1)^2}{z}$ for all $z>0$. Taking $z=\frac{x}{y}$ completes the proof of \eqref{ineq:log}.

In proving \eqref{ineq:log_II}, we note that $0\le f'(x)\le h'(x)$ for $x>0$ with $f(x)=\frac{x}{1+x}$ and  $h(x)=\log(1+x)$. Then, for all $x\le y$, we have 
$$
f(y)-f(x)=\int_x^y f'(s)\,{\rm d}s\le \int_x^y h'(s)\,{\rm d}s=h(y)-h(x). 
$$
\end{proof}

A Moser--Trudinger's inequality will be required on polygonal domains \cite{Bonilla_GS_2024}.  

\begin{theorem}  Let $x_h\in X_h$ with $x_h>0$. Then there exists a constant $C_\Omega>0$, independent of $h$, such that
\begin{equation}\label{Moser-Trudinger}
\int_\Omega i_h e^{x_h(\x)}\,\dx\le C_\Omega (1+ \|\nabla x_h\|^2_{L^2(\Omega)}) e^{\displaystyle C_\Omega \Big( \|\nabla x_h\|^2_{L^2(\Omega)}+\|x_h\|_{L^1(\Omega)}\Big)}.
\end{equation}
\end{theorem}

\section{Proof of Theorem \ref{Th:main}}
The proof of existence of generalized solutions, as stated in Theorem \ref{Th:main}, has been divided into five parts: analysis of solvability, obtainment of lower and upper bounds, derivation of \emph{a priori} estimates, establishment of weak and strong convergences, and finally, passage to the limit.   

\subsection{Solvability} 
 Since time has been kept continuous, system \eqref{eq:u_h}-\eqref{eq:v_h} is nothing more than a system of ordinary differential equations. Therefore the existence and uniqueness of discrete solutions can be inferred from Cauchy--Lindelöf's theorem. 

\begin{lemma} There exists a unique pair solution $(u_h, v_h)\in C^1([0,\infty); X_h)$ to \eqref{eq:u_h}-\eqref{eq:v_h}.   
\end{lemma}
 \begin{proof}
We define the vectors $\u_h=(u_i)_{i\in I}\in\R^I$ and $\v_h=(v_i)\in \R^I$ as being the coefficients in the linear combinations of $u_h=\sum_{i\in I} u_i\varphi_{\a_i}$ and $v_h=\sum_{i\in I}v_i\varphi_{\a_i}$, respectively.  Thus, for each $i\in I$, one can extract the corresponding differential equation from \eqref{eq:u_h}-\eqref{eq:v_h} by choosing $\bar u_h=\varphi_{\a_i}$ in \eqref{eq:u_h} and $\bar v_h=\varphi_{\a_i}$ in \eqref{eq:v_h}, which allows us to write

\begin{equation}\label{sec4.1:lab1}
\begin{array}{rcl}
\displaystyle
u'_i &=&\displaystyle - \|\varphi_{\a_i}\|_{L^1(\Omega)}^{-1}\sum_{\mbox{\tiny $\begin{array}{c} j\in I(\Delta_{\a_i})  \\ i<j\end{array}$}} u_j(\nabla\varphi_{\a_j}, \nabla \varphi_{\a_i})
\\
&&
\displaystyle
+\|\varphi_{\a_i}\|_{L^1(\Omega)}^{-1}\sum_{\mbox{\tiny$\begin{array}{c} j\in I(\Delta_{\a_i})  \\ i<j\end{array}$}}\tau_{ji}(u_h) \delta_{ji} \log(v_h)(\nabla \varphi_{\boldsymbol{a}_j},\nabla\varphi_{\boldsymbol{a}_i})
\\
&&
\displaystyle
-\|\varphi_{\a_i}\|_{L^1(\Omega)}^{-1}\sum_{\mbox{\tiny$\begin{array}{c} j\in I(\Delta_{\a_i})  \\ i<j\end{array}$}}\beta_{ji}(u_h, v_h) (u_j- u_i) 
\\
&:=&\|\varphi_{\a_i}\|_{L^1(\Omega)}^{-1}(- F_1^u(\u_h, \v_h)+F_2^u(\u_h,\v_h)-F_3^u(\u_h, \v_h))
\\
&:=&\mathcal{F}^u_i(\u_h, \v_h)
\end{array}
\end{equation}
and
\begin{equation}\label{sec4.1:lab2}
\begin{array}{rcl}
\displaystyle
v'_i&=&\displaystyle - \|\varphi_{\a_i}\|_{L^1(\Omega)}^{-1} \sum_{\mbox{\tiny$\begin{array}{c} j\in I(\Delta_{\a_i})  \\ i<j\end{array}$}} v_i (\nabla\varphi_{\a_j},\nabla \varphi_{\a_i})- u_i v_i 
\\
&:=&\|\varphi_{\a_i}\|_{L^1(\Omega)}^{-1}(-F_1^v(\u_h, \v_h)-F_2^v(\u_h,\v_h))
\\
&:=&\mathcal{F}^v_i(\u_h, \v_h).
\end{array}
\end{equation}

If one further defines $\boldsymbol{y}_h=(\u_h, \v_h)$, $\boldsymbol{\mathcal{F}}^u=(\mathcal{F}_i^u)_{i\in I}$ and $\boldsymbol{\mathcal{F}}^v=(\mathcal{F}^v_i)_{i\in I}$,  system \eqref{eq:u_h}-\eqref{eq:v_h} can be rewritten as  
\begin{equation}\label{IVP_for_y}
\left\{
\begin{array}{rcl}
\boldsymbol{y}_h'&=&\boldsymbol{\mathcal{F}} (\boldsymbol{y}_h),
\\
\boldsymbol{y}_h(0)&=&\boldsymbol{y}_{0h},
\end{array}
\right.
\end{equation}
where $\boldsymbol{\mathcal{F}}=(\boldsymbol{\mathcal{F}}^u, \boldsymbol{\mathcal{F}}^v)$ and $\boldsymbol{y}_{0h}=(\u_{0h},\v_{0h})$ with $\u_{0h}=(u_{0i})$ and $\v_{0h}=(u_{0i})$ being such that $u_{0h}=\sum_{i\in I} u_{0i}\varphi_{\a_i}$ and $v_{0h}=\sum_{i\in I} v_{0i}\varphi_{\a_i}$.

By inspecting $\boldsymbol{\mathcal{F}}$, one readily sees that it is continuous on $\mathcal{O}:= \mathcal{O}_u\times\mathcal{O}_v$, where $\mathcal{O}_u=\{\u_h\in\R^{I}\,:\, 1+u_h>0\}$ and $\mathcal{O}_v=\{\v_h\in\R^I :\, \v_h>0\}$. Then Peano's theorem assures existence of a solution on $[0,T]$ for certain $T>0$. 

For uniqueness we now concentrate our attention on proving that the right-hand side of \eqref{sec4.1:lab1} and \eqref{sec4.1:lab2} are locally Lipschitzian; a task that will be somewhat involved. 

For $\Xi= u$ or $v$, let $m_\Xi, M_\Xi\in\R$ be such that $-\frac{1}{2}<m_u < \min_{i\in I} u_{0i}$, and $0<m_v<\min_{i\in I} v_{0i}$ and $M_u>|\u_0|_{\infty}$ and $M_v>|\v_0|_{\infty}$, where $|\cdot|_{\infty}$ stands for the maximum norm on $\R^I$.  Consider   
$$
\mathcal{B}(\boldsymbol{y}_0)=\mathcal{B}(\u_{0h}; m_u, M_u)\times \mathcal{B}(\v_{0h}; m_v, M_v),
$$
where
$$
\mathcal{B}(\u_{0h}; m_u, M_u)=\{\u_h\in\R^{I}\,:\,  m_u\le \u_h \le M_u\}
$$
and
$$
\mathcal{B}(\v_{0h}; m_v, M_b) = \{\v_h\in\R^{I}\,:\ m_v \le \v_h\le M_v \}.
$$

We analyze separately: 
\begin{itemize}
\item As $F^\Xi_1$, for $\Xi=u$ or $v$, in \eqref{sec4.1:lab1} and \eqref{sec4.1:lab2}, respectively, are linear, they are straightforwardly seen to be Lipzchitzian. 

\item The term $F_2^v$ is quadratic and hence Lipzchitzian as well. 

\item The term $F_2^u$,  being the chemotactic one, requires some elaborations. For two pairs $(\u_h, \v_h)\in\mathcal{B}(\boldsymbol{y}_0)$ and $(\bar \u_h, \bar \v_h)\in\mathcal{B}(\boldsymbol{y}_0)$, write    
$$
\begin{array}{rcl}
F^u_2(\u_h, \v_h)-F^u_2(\bar\u_h, \bar\v_h)&=&\displaystyle\sum_{\mbox{\tiny$\begin{array}{c} j\in I(\Delta_{\a_i})  \\ i<j\end{array}$}}\Big(\tau_{ji}(u_h)- \tau_{ji}(\bar u_h)\Big) \delta_{ji} \log(v_h)(\nabla \varphi_{\boldsymbol{a}_j},\nabla\varphi_{\boldsymbol{a}_i})
\\
&&
\displaystyle 
+\sum_{\mbox{\tiny$\begin{array}{c} j\in I(\Delta_{\a_i})  \\ i<j\end{array}$}}\tau_{ji}(\bar u_h) \Big(\delta_{ji} \log(v_h)- \delta_{ji} \log(\bar v_h)\Big) (\nabla \varphi_{\boldsymbol{a}_j},\nabla\varphi_{\boldsymbol{a}_i})
\\
&=&T_1+ T_2.
\end{array}
$$

\begin{itemize}
\item To treat $T_1$, one needs to control the difference $\tau_{ji}(u_h)- \tau_{ji}(\bar u_h)$. According to the definition of $\tau_{ji}(\cdot)$ in \eqref{def:tau_ij} there are four expressions that such a difference might take. We only detail two of these expressions, because the others are handled  in pretty much the same way. Then:

\textbf{Case} $u_i\not=u_j$ and $\bar u_i=\bar u_j$: From \eqref{def:tau_ij}, we note that
$$
\begin{array}{rcl}
\tau_{ji}(u_h)-\tau_{ji}(\bar u_h) &=&\displaystyle \left|\Lambda_{ij}(u_h)\frac{\delta_{ji} \log(1+u_h)}{\delta_{ij}\frac{1}{1+u_h}}- \bar u_i\right|
\\ [15pt]
&\le&\displaystyle
\left|\Big(\Lambda_{ij}(u_h)-\Lambda_{ij}(\bar u_h)\Big) \frac{\delta_{ji} \log(1+  u_h)}{\delta_{ij}\frac{1}{1+ u_h}}\right|
\\
&&\displaystyle
+\left|\Lambda_{ij}(\bar u_h) \left(\frac{\delta_{ji} \log(1+  u_h)}{\delta_{ij}\frac{1}{1+ u_h}}- \frac{\bar u_i}{\Lambda_{ij}(\bar u_h)}\right)\right|
\\
&:=&T_{11}+T_{12}, 
\end{array}
$$ 
where recall $\Lambda_{ij}(u_h)=\frac{1}{2}(\frac{u_i}{1+u_i}+\frac{u_j}{1+u_j})$. By the mean value theorem, we have that there are $\lambda_i, \lambda_j\in (0,1)$ such that  
$$
\begin{array}{rcl}
\Lambda_{ij}(u_h)-\Lambda_{ij}(\bar u_h)&=&\displaystyle\frac{1}{2}\frac{1}{(\lambda_i(1+u_i) + (1-\lambda_i)(1+ \bar u_i) )^2} (u_i-\bar u_i)
\\
&&\displaystyle
+\frac{1}{2}\frac{1}{(\lambda_j(1+u_j) + (1-\lambda_j)(1+ \bar u_j) )^2} (u_j-\bar u_j)
\\
&\le&\displaystyle
2 (|u_i-\bar u_j|+|u_j-\bar u_j|),
\end{array}
$$
since $\bar u_i, u_i>- \frac{1}{2}$ and, by \eqref{ineq:log}, that  
$$
\frac{\delta_{ji} \log(1+ \bar u_h)}{\delta_{ij}\frac{1}{1+\bar u_h}}\le \sqrt{1+\bar u_i}\sqrt{1+\bar u_j}\le 1+M_u.
$$
It then follows  that  
$$
T_{11}\le 2 \Big(1+M_u\Big) \Big( |u_i-\bar u_i|+ |u_j-\bar u_j| \Big). 
$$
Observe now that, as $\bar u_i= \bar u_j$, it gives $\Lambda_{ij}(\bar u_h)= \frac{\bar u_i}{1+\bar u_i}$. Then the mean value theorem, applied to $T_{12}$, shows that   
 $$
\begin{array}{l}
\displaystyle
\frac{\delta_{ji} \log(1+  u_h)}{\delta_{ij}\frac{1}{1+ u_h}}- \frac{\bar u_i}{\Lambda_{ij}(\bar u_h)}
\\
=\displaystyle\left|\frac{(1+ u_j) (1+ u_i)}{\theta (1+  u_i)+(1-\theta)(1+  u_j)}-(1+\bar u_i)\right|
\\
=\displaystyle\left|\frac{(1+ u_j) (1+ u_i)}{\theta (1+  u_i)+(1-\theta)(1+  u_j)}-\frac{(1+\bar u_i)(1+ \bar  u_j)}{\theta (1+ \bar u_i)+(1-\theta)(1+ \bar u_j)}\right|
\\
\le\displaystyle\frac{1}{\theta (1+ u_i)+(1-\theta)(1+ u_j)} \Big((1+ u_j)| u_i- \bar u_i|+(1+\bar u_i)|u_j-\bar u_j|\Big)
\\
 + \displaystyle (1+  \bar u_j) (1+ \bar u_i)\left|\frac{1}{\theta (1+ u_i)+(1-\theta)(1+u_j)}-\frac{1}{\theta (1+   \bar u_i)+(1-\theta)(1+  \bar u_j)}\right|
 \\
 \le \displaystyle
 4(1+M_u)\Big(|\bar u_i- u_i|+|\bar u_j - u_j|\Big),
\end{array}
$$
where in the second inequality we used Young's inequality $a^{\tilde\theta} b^{1-\tilde\theta}\le \tilde\theta a + (1-\tilde\theta) b$.
As $\Lambda_{ij}(\bar u_h)< 2 M_u$, this implies that 
\begin{equation}\label{sec4.1:lab3}
T_{12}\le 8 M_u (1+M_u)\Big(|\bar u_i- u_i|+|\bar u_j - u_j|\Big).
\end{equation}

Further invoking \eqref{inv_local:WlpToWmq} gives  
$$
\begin{array}{rcl}
(\nabla\varphi_{\a_j},\nabla\varphi_{\a_i})&\le&\displaystyle \sum_{T\in \Delta_{\a_i}\cap\Delta_{\a_j}} \|\nabla\varphi_{\a_i}\|_{L^2(T)} \nabla\varphi_{\a_j}\|_{L^2(T)}
\\
&\le&\displaystyle
C_{\rm inv}  \sum_{T\in \Delta_{\a_i}\cap\Delta_{\a_j}}   \|\varphi_{\a_i}\|_{L^\infty(T)} \|\varphi_{\a_j}\|_{L^\infty(T)}
\\
&\le& C_{\rm inv} C_{\cap}:= K_\cap,
\end{array}
$$
where $C_{\cap}=\max_{i,j\in I} {\rm card}\{ T : T\in \Delta_{\a_i}\cap\Delta_{\a_j}\}$. Thus we arrive at 
\begin{equation}\label{sec4.1:lab4}
T_1\le\displaystyle 2 K_\cap \max\{1,  4 M_u\} (1+M_u)\frac{M_v}{m_v}  \sum_{j\in I(\Delta_{\a_i})}  \Big(|\bar u_i- u_i|+|\bar u_j - u_j|\Big),
\end{equation}
since 
$$
\delta_{ji} \log(v_h)=\frac{1}{\gamma v_i + (1-\gamma) v_j} ( v_i - v_j)\le  \frac{M_v}{m_v}
$$
for $\gamma\in(0,1)$. 

  \textbf{Case} $u_i\not=u_j$ and $\bar u_i\not=\bar u_j$: Now 
$$
\begin{array}{rcl}
\tau_{ji}(n_h)-\tau_{ji}(\bar n_h) &=&\displaystyle \left|\Lambda_{ij}(u_h)\frac{\delta_{ji} \log(1+u_h)}{\delta_{ij}\frac{1}{1+u_h}}-\Lambda_{ij}(\bar u_h)\frac{\delta_{ji} \log (1+\bar u_h)}{\delta_{ij}\frac{1}{1+\bar u_h}}\right|
\\ [15pt]
&\le&\displaystyle
\left|\Big(\Lambda_{ij}(u_h)-\Lambda_{ij}(\bar u_h)\Big) \frac{\delta_{ji} \log(1+  u_h)}{\delta_{ij}\frac{1}{1+ u_h}}\right|
\\
&&\displaystyle
+\left|\Lambda_{ij}(\bar u_h) \left(\frac{\delta_{ji} \log(1+  u_h)}{\delta_{ij}\frac{1}{1+ u_h}}-\frac{\delta_{ji} \log(1+ \bar u_h)}{\delta_{ij}\frac{1}{1+\bar u_h}}\right)\right|
\\
&:=&T_{11}+T_{12}.
\end{array}
$$ 
We only need to control the new term $T_{12}$. Let $ g: [m_u, M_u]^2\to \R$ such that 
$$
g(x,y)=\left\{
\begin{array}{ccl}
\displaystyle
\frac{\log(1+y)-\log (1+x)}{\frac{1}{1+x}-\frac{1}{1+y}}&\mbox{ if }& x\not=y,
\\
 1+x&\mbox{ if }& x=y.
\end{array}
\right.
$$
We want to prove that $g$ is Lipschizian on $[m_u, M_u]^2$. We do so by controlling  $\partial_x g$ and $\partial_y g$ in $W^{1,\infty}([m_u, M_u]^2)$. Indeed,
$$
\partial_x g(x,y)= \frac{(1+y)[(1+y)\log(\frac{1+y}{1+x})-x+y]}{(x-y)^2}=\frac{(1-\theta)(1+y)}{\theta(1+y)+(1-\theta)(1+x)},
$$ 
with $\theta\in(0,1)$ depending on $x$ and $y$. If $m_u\le x,y\le M_u$, it follows that 
$$
\|\partial_x g(x,y)\|_{L^\infty([m_u, M_u]^2)}\le 2 (1+M_u).
$$
Analogously, one finds  
$$
\|\partial_y g(x,y)\|_{L^\infty([m_u, M_u]^2)}\le2(1+M_u).
$$
Morrey's inequality implies the Lipschizian property for $g$. Consequently inequality \eqref{sec4.1:lab3} holds and hence \eqref{sec4.1:lab4} does as well.

\item For $T_2$, it should be noted that
$$
\begin{array}{rcl}
\delta_{ji} \log(v_h)- \delta_{ji} \log(\bar v_h)&=&\log  v_i -\log \bar v_i+ \log  v_j-\log \bar v_j
\\
&\le& \dfrac{1}{m_v}\Big(|v_i-\bar v_i|+|v_j-\bar v_j|\Big).
\end{array}
$$ 
Furthermore if one assumes from \eqref{def:tau_ij} that $u_i\not=u_j$ holds, there follows that, for $\theta\in(0,1)$,   
$$
\begin{array}{rcl}
\tau_{ji}(\bar u_h)&=& \Lambda_{ij}(\bar u_h) \dfrac{\delta_{ji} \log(1+ \bar u_h)}{\delta_{ij}\frac{1}{1+\bar u_h}}
\\[3ex]
&\le& 2 M_u (1+\bar u_i) (1+\bar u_j) \dfrac{\delta_{ji}\log(1+\bar u_h)}{\delta_{ji}(1+\bar u_h)}.
\\
&\le& 2 M_u \dfrac{(1+\bar u_i) (1+\bar u_j)}{\theta (1+\bar u_j) + (1-\theta)(1+ \bar u_i)}\le 4 M_u( 1+M_u)^2;
\end{array}
$$
it is obvious for $u_i=u_j$. These two last bounds lead to 
\begin{equation}\label{sec4.1:lab5}
T_2\le 4 K_\cap M_u \frac{(1+M_u)^2}{m_v} \sum_{j\in I(\Delta_{\a_i})} \Big(|v_i-\bar v_i|+|v_j-\bar v_j|\Big).
\end{equation}

From \eqref{sec4.1:lab4} and \eqref{sec4.1:lab5}, one can find $C_{\rm m_\#, M_\#, \cap }=C(m_u, M_u, m_v, M_v, K_\cap)>0$ such that     
$$
\begin{array}{rcl}
|F^u_2(\u_h, \v_h)-F^u_2(\bar\u_h, \bar\v_h)|&\le& \displaystyle
C_{ m_\#, M_\#, \cap } \left(\sum_{j\in I(\Delta_{\a_i})}\Big(|v_i-\bar v_i|+|v_j-\bar v_j|\Big) \right.
\\
&&\displaystyle
\left.+\sum_{j\in I(\Delta_{\a_i})}  \Big(|\bar u_i- u_i|+|\bar u_j - u_j|\Big)\right).
\end{array}
$$

\end{itemize}

\item To prove that $F_3^u(\cdot, \cdot)$ is Lipschitzian, we proceed in the following manner. On comparing, we obtain  
$$
\begin{array}{rcl}
F_3^u(\u_h, \v_h)-F_3^u(\bar\u_h, \bar\v_h)&=&\displaystyle\sum_{j\in I(\Delta_{\a_i})}\Big(\beta_{ji}(u_h, v_h)- \beta_{ji}(\bar u_h,\bar v_h) \Big) (u_j- u_i)
\\
&&\displaystyle+\sum_{j\in I(\Delta_{\a_i})}\beta_{ji}(\bar u_h, \bar v_h) \Big( (u_j- u_i)- (\bar u_j- \bar u_i)\Big)
\\
&:=&S_1+S_2.
\end{array}
$$
Thus
$$
\begin{array}{rcl}
\beta_{ji}(u_h, v_h)- \beta_{ji}(\bar u_h, \bar v_h)&=&\max\{\alpha_{\a_i}(u_h) f_{ij}(u_h, v_h), \alpha_{\a_j}(u_h)f_{ji}(u_h,v_h), 0\}
\\
&&-\max\{\alpha_{\a_i}(\bar u_h) f_{ij}(\bar u_h,\bar v_h), \alpha_{\a_j}(\bar u_h)f_{ji}(\bar u_h, \bar v_h), 0\}
\\
&\le&\max\{\alpha_{\a_i}(u_h) f_{ij}(u_h, v_h)-\alpha_{\a_i}(\bar u_h) f_{ij}(\bar u_h, \bar v_h ), 
 \\
&&\hspace{1.5cm}\alpha_{\a_j}(u_h)f_{ji}(u_h, v_h)-\alpha_{\a_j}(\bar u_h)f_{ji}(\bar u_h, \bar v_h), 0\},
\end{array}
$$ 
where recall 
$$
f_{ij}(u_h, v_h)=\left\{
\begin{array}{ccl}
\displaystyle
\delta_{ji} \log v_h \Big[\frac{u_i}{\delta_{ji} u_h}- \frac{\tau_{ji}(u_h)}{\delta_{ji} u_h}\Big](\nabla \varphi_{\a_j}, \nabla \varphi_{\a_i})&\mbox{ if }& x_j\not=x_i,
\\
0&\mbox{ if }&x_j=x_i.
\end{array}
\right.
$$

Let us define $f_{ij}(u_h, v_h)= \delta_{ji} \log v_h \Big[\frac{\tau_{ji}(u_h)}{\delta_{ji} u_h}-\frac{u_i}{\delta_{ji} u_h}\Big](\nabla\varphi_{\a_j}, \nabla\varphi_{\a_i})=\bar f_{ij}(u_h, v_h)/\delta_{ji} u_h$ and write
\begin{equation}\label{sec4.1:lab6}
\begin{array}{rcl}
\alpha_{\a_i}(u_h) f_{ij}(u_h, v_h)-\alpha_{\a_i}(\bar u_h) f_{ij}(\bar u_h, \bar v _h )&=&f_{ij}(u_h, v_h)(\alpha_{\a_i}(u_h)-\alpha_{\a_i}(\bar u_h))
\\
&&+ \alpha_{\a_i}(\bar u_h) ( f_{ij}(u_h,v_h)- f_{ij}(\bar u_h, \bar v_h))
\\
&=&\displaystyle f_{ij}(u_h, v_h)(\alpha_{\a_i}(u_h)-\alpha_{\a_i}(\bar u_h))
\\
&&\displaystyle+\frac{\alpha_{\a_i}(\bar u_h)}{\delta_{ji} u_h }(\bar f_{ij}(u_h, v_h) - \bar f_{ij}(\bar u_h, \bar v_h ))
\\
&&\displaystyle
+\alpha_{\a_i}(\bar u_h) \bar f_{ij}(\bar u_h, \bar v_h )\Big(\frac{1}{\delta_{ji} u_h }-\frac{1}{\delta_{ji} \bar u_h }\Big).
\end{array}
\end{equation}   
Let 
$$
\widetilde\alpha_{\a_i}(x_h) = \left\{
\begin{array}{cc}  
\left[
\frac{{\sum_{j\in I(\Delta_{\a_i})} [\![\nabla x_h]\!]_{ij}}}{\sum_{j\in I(\Delta_{\a_i})} 2\{\!\!\{|\nabla v_h \cdot \hat{\boldsymbol{r}}_{ij}|\}\!\!\}_{ij}}
\right]^q & \text{if } 
\sum_{j\in I(\Delta_{\a_i})} \{\!\!\{|\nabla x_h\cdot \hat{\boldsymbol{r}}_{ij}|\}\!\!\}_{ij} \neq 0, 
\\
0 & \text{otherwise},
\end{array}
\right. 
$$
which is exactly as $\alpha_{\a_i}$ but without using the positive part. 
As
$$
|\alpha_{\a_i}(u_h)-\alpha_{\a_i}(\bar u_h)|\le |\widetilde\alpha_{\a_i}(u_h)-\widetilde\alpha_{\a_i}(\bar u_h)|, 
$$
we use
$$
\widetilde\alpha_{\a_i}(u_h)^{\frac{1}{q}}=\frac{\left|\displaystyle\sum_{k\in I(\Delta_{\a_i})} \Big(\frac{u_k - u_i}{|\boldsymbol{r}_{ik}|} + \frac{u_k^{\rm sym} - u_i}{|\boldsymbol{r}_{ik}^{\rm sym}|}\Big)\right|}{\displaystyle\sum_{j\in I(\Delta_{\a_i})} \left(\frac{|u_k-u_i|}{|\boldsymbol{r}_{ij}|}+\frac{|u_k^{\rm sym}-u_i|}{|\boldsymbol{r}_{ij}^{\rm sym}|}\right)}.
$$
In view of \eqref{ineq:sum}, we have, on noting that $|x^q-y^q|\le q |x-y|$ holds for all $x,y\in[0,1]$,  that \cite[Th. 6.1]{Badia_Bonilla_2017} 
$$
|\alpha_{\a_i}(u_h)-\alpha_{\a_i}(\bar u_h)|\le 2 q\frac{\Big|{\displaystyle\sum_{k\in I(\Delta_{\a_i})} \Big(\frac{(u-\bar u)_k - (u - \bar u)_i}{|\boldsymbol{r}_{ik}|} + \frac{(u-\bar u)_k^{\rm sym} - (u- \bar u)_i}{|\boldsymbol{r}_{ik}^{\rm sym}|}\Big)}\Big|}{\displaystyle\sum_{k\in I(\Delta_{\a_i})} \left(\frac{|u_k-u_i|}{|\boldsymbol{r}_{ik}|}+\frac{|u_k^{\rm sym}-u_i|}{|\boldsymbol{r}_{ik}^{\rm sym}|}\right)}.
$$
From the quasi-uniformity for $\mathcal{T}_h$, we know that there exists $\rho>0$ such that $\rho h\le |\boldsymbol{r}_{ik}|, |\boldsymbol{r}^{\rm sym}_{ik}|\le h $. Thus, one can find $C_{q,\rho}=C(q,\rho)>0$, independent of $h$, such that 
$$
\begin{array}{rcl}
|\alpha_{\a_i}(u_h)-\alpha_{\a_i}(\bar u_h)|&\le& \displaystyle C_{q,\rho} \frac{\displaystyle\sum_{k\in I(\Delta_{\a_i})} \Big(|(u-\bar u)_k - (u - \bar u)_i| + |(u-\bar u)_k^{\rm sym} - (u- \bar u)_i|\Big)}{\displaystyle\sum_{k\in I(\Delta_{\a_i})} \Big(|u_k-u_i|+|u_k^{\rm sym}-u_i|\Big)}
\\
\displaystyle &\le&\displaystyle C_{q,\rho} \frac{\displaystyle\sum_{k\in I(\Delta_{\a_i})} \Big(|u_k-\bar u_k|+|u_i-\bar u_i| \Big)}{|u_j-u_i|},
\end{array}
$$
where we used that there is $\gamma\in(0,1)$ such that $\a_{k_1}, \a_{k_2}\in\mathcal{N}_h$ such that $\a_{k}^{\rm sym}=\gamma \a_{k_1}+ (1-\gamma)\a_{k_2}$ with  $u^{\rm sym}_k=\gamma u_{k_1}+(1-\gamma) u_{k_2}$ and $\bar u^{\rm sym}_k=\gamma \bar u_{k_1}+(1-\gamma) \bar u_{k_2}$. Additionally, note that $ |f_{ij}(u_h, v_h)|\le K_\cap[ 1+(1+M_u)\frac{M_v}{m_v}]$,
since 
$$
\begin{array}{rcl}
\displaystyle
\frac{\tau_{ji}(u_h)}{\delta_{ji} u_h}-\frac{u_i}{\delta_{ji} u_h}&=&\displaystyle
\frac{1}{2} \frac{1}{\delta_{ji} u_h} \frac{u_i}{1+u_i}  \Big(\frac{(1+u_j)(1+u_i)}{\mu(1+u_i)+(1-\mu)(1+u_j)}- (1+u_i) \Big) 
\\
&&\displaystyle
+ \frac{1}{2}  \frac{1}{\delta_{ji} u_h} \Big( \frac{u_j}{1+u_j}  \frac{(1+u_j)(1+u_i)}{\mu(1+u_i)+(1-\mu)(1+u_j)}- u_i  \Big)
\\
&:= &B_1+B_2,
\end{array}
$$
with
\begin{equation}\label{sec4.1:B1}
\begin{array}{rcl}
B_1&=&\displaystyle
\frac{1}{2} \frac{1}{\delta_{ji} u_h} \frac{u_i}{1+u_i}  \Big(\frac{(1+u_j)(1+u_i)}{\mu(1+u_i)+(1-\mu)(1+u_j)}
\\
&&\displaystyle
\hspace{4cm}- \frac{(1+u_j)(1+u_i)}{\mu(1+u_j)+(1-\mu)(1+u_j)} \Big) 
\\
&=&\displaystyle
\frac{1}{2} \frac{u_i (1+u_j)}{\delta_{ji} u_h}   \Big(\frac{1}{\mu(1+u_i)+(1-\mu)(1+u_j)}
\\
&&\displaystyle
\hspace{4cm}- \frac{1}{\mu(1+u_j)+(1-\mu)(1+u_j)} \Big)
\\
&=&\displaystyle
\frac{1}{2} \frac{u_i (1+u_j) }{\delta_{ji} u_h} \frac{\mu \delta_{ji} u_h}{[\mu(1+u_i)+(1-\mu)(1+u_j)] (1+u_j)} 
\\
&=&\displaystyle
\frac{1}{2}  \frac{\mu u_i}{\mu(1+u_i)+(1-\mu)(1+u_j)} 
\end{array}
\end{equation}
and
\begin{equation}\label{sec4.1:B2}
\begin{array}{rcl}
B_2&=&\displaystyle
\frac{1}{2}  \frac{1}{\delta_{ji} u_h} \Big( \frac{u_j}{1+u_j}  \frac{(1+u_j)(1+u_i)}{\mu(1+u_i)+(1-\mu)(1+u_j)} - u_j+  (u_j - u_i)  \Big)
\\
&=&\displaystyle
\frac{1}{2}  \frac{(1-\mu) u_j}{\mu(1+u_i)+(1-\mu)(1+u_j)} + \frac{1}{2},
\end{array}
\end{equation}
which give 
$$
\left|\frac{\tau_{ji}(u_h)}{\delta_{ji} u_h}-\frac{u_i}{\delta_{ji} u_h}\right|\le  1+M_u .  
$$
This leads to 
\begin{equation}\label{sec4.1:lab7}
\begin{array}{rcl}
|f_{ij}(u_h, v_h)(\alpha_{\a_i}(u_h)&-&\alpha_{\a_i}(\bar u_h))|
\\
&\le& \displaystyle C_{q,\rho, \cap} \Big[ 1+(1+M_u)\frac{M_v}{m_v}\Big]\frac{\displaystyle\sum_{k\in I(\Delta_{\a_i})} \Big(|u_k-\bar u_k|+|u_i-\bar u_i| \Big)}{|u_j-u_i|}. 
\end{array}
\end{equation}
By Lemma \ref{lm:alpha_i}, we know that $|\alpha_{\a_i}|\le 1$. Thus it is easily seen from a similar argument that $F^u_2$ is Lipschitzian that 
\begin{equation}\label{sec4.1:lab8}
\begin{array}{rcl}
\displaystyle
\frac{\alpha_{\a_i}(\bar u_h)}{\delta_{ji} u_h }(\bar f_{ij}(u_h, v_h) - \bar f_{ij}(\bar u_h, \bar v_h ))&\le&\displaystyle C_{m_\#, M_\#, \cap} \left( \frac{|\bar u_i- u_i|+|\bar u_j - u_j|}{|u_j-u_i|}\right.
\\
&&\displaystyle
\hspace{2cm}\left.+\frac{|v_i-\bar v_i|+|v_j-\bar v_j|}{|u_j-u_i|}\right),
\end{array}f_{ij}
\end{equation}
where $C_{m_\#, M_\#, \cap}=C(m_u, M_u, m_v, M_v, \cap)$.  Next we write
\begin{align}\label{sec4.1:lab9} 
\displaystyle
\alpha_{\a_i}(\bar u_h) \bar f_{ij}(\bar u_h)\Big(\frac{1}{\delta_{ji} u_h}&-\frac{1}{\delta_{ji} \bar u_h}\Big)
\nonumber
\\
&=\alpha_{\a_i}(\bar u_h) \frac{\bar f_{ij}(\bar u_h)}{(u_j-u_i)(\bar u_j-\bar u_i)} \Big((\bar u_i- u_i)+(\bar u_j - u_j)\Big)
\\
& \displaystyle
\le K_\cap\Big[ 1+(1+M_u)\frac{M_v}{m_v}\Big] \frac{|\bar u_i- u_i|+|\bar u_j -  u_j|}{|u_j-u_i|}.
\nonumber
\end{align}

Plugging \eqref{sec4.1:lab7}, \eqref{sec4.1:lab8}, \eqref{sec4.1:lab9} into \eqref{sec4.1:lab6} yields 
\begin{align*}
|\beta_{ji}(u_h, v_h)&- \beta_{ji}(\bar u_h,\bar v_h)|
\\
&\le \frac{C_{q, \rho, m_\#, M_\#, \cap}}{|u_j-u_i|} \left(\sum_{k\in I(\Delta_{\a_i})} \Big(|u_k-\bar u_k|+|u_i-\bar u_i|\Big)
+|v_j-\bar v_j|+|v_i-\bar v_i|\right)
\end{align*}
and hence, if $K_\Delta=\max_{i\in I} I(\Delta_{\a_i})$,   
\begin{equation}\label{sec4.1:lab10}
S_1\le \frac{C_{q, \rho, m_\#, M_\#, \cap, \Delta}}{|u_j-u_i|} \left(\sum_{j\in I(\Delta_{\a_i})}\Big[\Big(|u_j-\bar u_j|+|u_i-\bar u_i|\Big)
+\Big(|v_j-\bar v_j|+|v_i-\bar v_i|\Big)\Big] \right).
\end{equation}

In regard to $S_2$, it can be trivially addressed as 
\begin{equation}\label{sec4.1:lab11}
S_2\le K_\cap[ 1+(1+M_u)\frac{M_v}{m_v}] \sum_{j\in I(\Delta_{\a_i})}\Big( |u_j- \bar u_j)+ | u_i- \bar u_i|\Big).
\end{equation}
In this last estimate, we utilized   
$$
\beta_{ji}(\bar u_h,\bar v_h)\le K_\cap[ 1+(1+M_u)\frac{M_v}{m_v}]. 
$$
From \eqref{sec4.1:lab10} and \eqref{sec4.1:lab11}, one can find $C_{q, m_\#, M_\#, \cap,\Delta}=C(q, m_\#, M_\#, K_\cap, K_\Delta)$ such that
$$
\begin{array}{rcl}
|F_3^u(\u_h, \v_h)&-&F_3^u(\bar\u_h, \bar\v_h)|
\\
&\le& \displaystyle
C_{q, \rho, m_\#, M_\#, \cap, \Delta} \Bigg(\sum_{j\in I(\Delta_{\a_i})}\Big[\Big(|u_j-\bar u_j|+|u_i-\bar u_i|\Big)
\\
 &&\hspace{5cm}+\Big(|v_i-\bar v_i|+|v_j-\bar v_i|\Big)\Big] \Bigg). 
\end{array}
$$
\end{itemize}
Finally, we conclude that  $\boldsymbol{\mathcal{F}}$ is Lipschitzian on $\mathcal{B}(\boldsymbol{y}_0)$ for the taxicab norm $\|\cdot\|_1$. To be more precise, by tracing back,  there exists $C_{\rm Lip}=C_{q, \rho, m_\#, M_\#, \cap,\Delta}>0$ such that
$$
\|\boldsymbol{\mathcal{F}}(\boldsymbol{y}_h)-  \boldsymbol{\mathcal{F}}(\bar{\boldsymbol{y}}_h)\|_1\le C_{\rm Lip} \|\boldsymbol{y}_h- \bar{\boldsymbol{y}}_h\|_1\quad \mbox{ for all }\quad \boldsymbol{y}_h\in\mathcal{B}(\boldsymbol{y}_0); 
$$  
we have thus completed the proof of uniqueness. 

\end{proof}

\begin{remark} 
From the theory of ordinary differential equations we learn that for every $y_{0h}\in\mathcal{O}$ there is a maximal solution $\boldsymbol{y}_h:[0,T_{\rm max})\to \R^{2I}$ of the initial-value problem \eqref{IVP_for_y}. The global well-posedness  of system \eqref{eq:u_h}-\eqref{eq:v_h} lies in the fact that the trajectory $(u_h(t),v_h(t))\in \mathcal{O}$ holds for $t\in[0,\infty)$. This outcome will naturally emerge from upper, lower, and $L^1(\Omega)$ bounds for $u_h$ and $v_h$. Further details will be developed subsequently in the next two sections.    
\end{remark}

\subsection{Lower and upper bounds} 
One fundamental property of the sequence of discrete solution pairs $\{(u_h, v_h)\}_{h>0}$ computed through system \eqref{eq:u_h}-\eqref{eq:v_h} is that they enjoy lower and upper bounds; more particularly, positivity  for both $(u_h,v_h)$, and a discrete maximum principle for $v_h$ only. For this, we choose $u_{0h}=j_h u_0$ and $v_{0h}=j_h v_0$, which satisfy, by definition of $j_h$ in \eqref{def:jh}, that    
$$
0\le u_{0h} \quad\mbox{ for all } \quad \x\in \Omega
$$
and
$$
0<v_{0h}\le \|v_{0}\|_{L^\infty(\Omega)}\quad\mbox{ for all } \quad \x\in \Omega.
$$
Further we obtain, by \eqref{err_LpW1p:jh} and a regularization argument,  that, as $h\to 0$,  
\begin{equation}\label{conv:u0h->u0}
u_{0h}\to u_0\quad\mbox{ in }\quad L^1(\Omega)
\end{equation}
and
\begin{equation}\label{conv:v0h->v0}
v_{0h}\to v_0\quad\mbox{ in }\quad L^\infty(\Omega).
\end{equation}
\begin{lemma}
Let $(u_h, v_h)$ be the solution to \eqref{eq:u_h}-\eqref{eq:v_h} with initial data $(u_{0h}, v_{0h} )$. Then, the following properties hold for all $t \in [0, \infty)$:
\begin{equation}\label{Positivity_uh}
0\le u_h(\x,t)\quad\mbox{ for all } \quad (\x,t)\in \Omega\times[0,\infty)
\end{equation}
and 
\begin{equation}\label{Positivity_and_DMaxP_vh}
0<v_h(\x,t)\le \|v_{0h}\|_{L^\infty(\Omega)} \quad\mbox{ for all } \quad (\x,t)\in \Omega\times[0,\infty).
\end{equation}
\end{lemma}
\begin{proof}  We divide the proof into three steps:

\paragraph{Step 1: Nonnegativity of $u_h$.}  Let $\underline{t}\in[0, T_{\rm max})$ such that there exists $\boldsymbol{a}_i\in\mathcal{N}_h$ for which $u_i(\bar t):=u_h(\bar t, \boldsymbol{a}_i)= 0$  is a local minimum. Assume, as $u_h\in C([0,T_{\rm max}); U_h)$, there is $\bar t\in (\underline{t},  T_{\rm max})$ such that $u_i(t)$ remains  a local minimum for all  $t\in (\underline{t}, \bar t)$, and moreover $u_{\underline{t}}<0$. For the sake of simplicity and with no loss in generality, one may order the indexes such that $i<j$ for all $ j\in I(\Delta_{\a_i})$ and suppose that the leftmost inequality in \eqref{Positivity_and_DMaxP_vh} is fulfilled over $(\underline{t}, \bar t)$; otherwise one can use a truncation argument for $v_h$  (e.g., replacing $v_h$ by its nodal positive part). Define $I^*(\Delta_{\a_i})=\{j\in I(\Delta_{\a_i}): u_j=u_i\}$ and let  $I^*_c(\Delta_{\a_i})$ denote  its complementary.  Then choosing $\bar u_h=\varphi_{\a_i}$ in \eqref{eq:u_h} yields 
$$
(\partial_t u_h, \varphi_{\a_i})_h+(\nabla u_h, \nabla \varphi_{\a_i})-(u_h\nabla i_h\log v_h, \nabla \varphi_{\a_i})_*+(B(u_h,v_h) u_h, \varphi_{\a_i})=0.
$$
It is easy to check from \eqref{New_KS_discretization} and \eqref{def:tau_ij} that
$$
\begin{array}{rcl}
(u_h\nabla i_h\log v_h, \nabla \varphi_{\a_i})_*&=&\displaystyle\sum_{j\in \mathds{I}^*_c(\Delta_{\a_i})} \Lambda_{ij}(u_h)\frac{\delta_{ji}\log(1+u_h)}{\delta_{ij}\frac{1}{1+u_h}} \delta_{ji}\log v_h (\nabla \varphi_{\boldsymbol{a}_j},\nabla\varphi_{\boldsymbol{a}_i})
\\
&&+\displaystyle \sum_{j\in I^*(\Delta_{\a_i})} u_i \delta_{ji}\log v_h (\nabla \varphi_{\boldsymbol{a}_j},\nabla\varphi_{\boldsymbol{a}_i})
\\
&=&\displaystyle\sum_{j\in I^*_c(\Delta_{\a_i})} \Lambda_{ij}(u_h)\frac{\delta_{ji}\log(1+u_h)}{\delta_{ij}\frac{1}{1+u_h}} \delta_{ji}\log v_h (\nabla \varphi_{\boldsymbol{a}_j},\nabla\varphi_{\boldsymbol{a}_i})
\\
&&-\displaystyle \sum_{j\in I^*_c(\Delta_{\a_i})} u_i \delta_{ji}\log v_h (\nabla \varphi_{\boldsymbol{a}_j},\nabla\varphi_{\boldsymbol{a}_i})+ u_i (\nabla i_h\log v_h,\nabla\varphi_{\boldsymbol{a}_i}).
\end{array}
$$
Further it follows on noting \eqref{Bu} and \eqref{def:beta^u_ji} that  
$$
(B(u_h, v_h) u_h, \varphi_{\a_i})=- \sum_{j\in I^*_c(\Delta_{\a_i})} 
\beta_{ji}(u_h,v_h)\delta_{ji} u_h.
$$

Upon compiling the above equalities and recalling \eqref{def:fij^u}, one is led to write 
\begin{align*}
(1,\varphi_{\a_i}) u_i'&\displaystyle+ (\nabla u_h, \nabla\varphi_{\a_i})+\sum_{j\in I^*_c(\Delta_{\a_i})} f_{ij}(u_h, v_h) \delta_{ij} u_h 
\\
&- \sum_{j\in I^*_c(\Omega_{\a_i})} \beta_{ji}(u_h, v_h)\delta_{ji} u_h+  u_i (\nabla i_h\log v_h,\nabla\varphi_{\boldsymbol{a}_i}) =  0.
\end{align*}
In view of \eqref{def:beta^u_ji} and  observing $\alpha_i(u_h)=1$ from Lemma \ref{lm:alpha_i}, it is  deduced that, for all $t\in(\underline{t}, \bar t)$ and $j\in \mathds{I}^*_c(\Delta_{\a_i})$, 
$$
f_{ji}-\beta_{ji}(u_h,v_h)\le 0 
$$
holds and therefore
$$
[f_{ji}-\beta_{ji}(u_h, v_h)] \delta_{ji} u_h\le 0,
$$
since $u_h$ reaches a local minimum on $\Delta_{\a_i}$. Further, from the weak acuteness of $\mathcal{K}_h$, we have \cite{Xu_Zikatanov_1999} that 
$$
(\nabla u_h, \nabla\varphi_{\a_i})\le 0.
$$
Then this  gives  
\begin{align*}
0\le (1,\varphi_{\a_i}) u_i' + u_i (\nabla i_h\log v_h,\nabla\varphi_{\boldsymbol{a}_i}) \le (1,\varphi_{\a_i}) u_i' - u_i \|\nabla i_h\log v_h\|_{L^2(\Omega)} \|\nabla\varphi_{\boldsymbol{a}_i}\|_{L^2(\Omega)} ,
\end{align*}
because of $u_i(t)\le 0$ for all $t\in(\underline{t}, \bar t)$, which, in turn, yields, after integrating over $(\underline{t}, \bar t)$,  the contradiction \emph{ex hypothesi}
$$
0= u_i(\underline{t})\le \exp\left(- \int_{\underline{t}}^{\bar t} \frac{\|\nabla i_h\log v_h\|_{L^2(\Omega)} \|\nabla \varphi_{\a_i}\|_{L^2(\Omega)}}{\|\varphi_{\a_i}\|_{L^1(\Omega)}} \, {\rm d}s\right) u_i(\bar t)< 0,
$$
The above time integral is well-defined since $ v_h > 0 $ on $ [\underline{t}, \bar{t}] $, which ensures that $ \log v_h $ is well-defined, and the mapping $ s \mapsto \|\nabla i_h \log v_h(s)\|_{L^2(\Omega)} $ is continuous (hence integrable) on $[\underline{t}, \bar{t}] $.

\paragraph{Step 2: Positivity of $v_h$.} It will now be shown that the leftmost inequality in \eqref{Positivity_and_DMaxP_vh} holds. For this purpose, let $\bar t \in (0,T_{\rm max})$  be the first time for which there exists $\boldsymbol{a}_i \in \mathcal{N}_h$ such that $v_i(\bar t) := v_h(\boldsymbol{a}_i,\bar t) = 0$ is a local minimum. By continuity of $v_h$, there exists $\underline{t} \in (0,\bar t)$ such that $v_i(t)$ remains a local minimum for all $t \in (\underline{t}, \bar t)$. Picking $\bar v_h=\varphi_{\a_i}$ in \eqref{eq:v_h}, there results
$$
(\partial_t v_h,\varphi_{\a_i})_h +(\nabla v_h,\nabla \varphi_{\a_i})+(u_h v_h, \varphi_{\a_i})_h=0. 
$$
As before 
$$
(\nabla v_h, \nabla\varphi_{\a_i})\le 0
$$
and hence
$$
 v_i'+ v_i u_i\ge 0
$$
which, after integration over $(\underline{t}, \overline{t})$, gives
$$
0 < v_i(\underline{t}) \le  e^{\int_{\underline{t}}^{\bar t} u_i(s)\, d{\rm s} }v_i (\bar t). 
$$
But this last relation cannot hold unless $v_i(\bar t)> 0$, again contradicting the choice of $\a_i\in\mathcal{N}_h$. 

\paragraph{Step 3: Discrete maximum principle for $v_h$.} The similar procedure applies for proving the rightmost inequality in \eqref{Positivity_and_DMaxP_vh}, i.e., a discrete maximum principle for $v_h$. Let us thus assume that $\underline{t}\in [0,T_{\rm max})$ is the first time such that there exists a local maximum $v_i(\underline{t})=\|v_{0h}\|_{L^\infty(\Omega)}$ at certain node $\a_i\in\mathcal{N}_h$, which evolves growing up still as a local maximum over $(\underline{t}, \overline{t})$, i.e.,  $v_i(t)>\|v_{0h}\|_{L^\infty(\Omega)}$ for all $t\in(\underline{t}, \bar t]$. Then substituting $\bar v_h=\varphi_{\a_i}$ into \eqref{eq:v_h} and using 
$$
(\nabla v_h, \nabla\varphi_{\a_i})\ge 0,
$$ 
we arrive  at 
$$
v_i' + v_i u_i \le 0.
$$
On noting that $v_i u_i\ge 0$ on $[\underline{t}, \bar t]$ from the positivity in \eqref{Positivity_uh} and \eqref{Positivity_and_DMaxP_vh} and  integrating over $(\underline{t}, \bar t)$, one finds  
$$
v_i(\bar t)\le v_i(\underline{t})=\|v_{0h}\|_{L^\infty(\Omega)},
$$
This results in a contradiction from our assumption about $\a_i\in\mathcal{N}_h$. Thus the proof is complete.
\end{proof}
\begin{remark}
At this point one can only assure that, if there is a blow up in finite time, it occurs when
$\|u_h(t)\|_{L^\infty(\Omega)}\to +\infty$  as $t\to T^-_{\rm max}$. Thus we need to get an upper bound for $u_h$. 
\end{remark}
\subsection{$L^1(\Omega)$ bounds} The $L^1(\Omega)$ bounds are immediately inherited from the structure of the chemotaxis and stabilizing terms. 
\begin{lemma}
For any $t \in [0, \infty)$, the following properties hold:
\begin{equation}\label{Mass-convervation-uh}
\|u_h(t)\|_{L^1(\Omega)}=\|u_{0h}\|_{L^1(\Omega)}
\end{equation}
and
\begin{equation}\label{L1-bound-vh}
\|v_h(t)\|_{L^1(\Omega)}\le \|v_{0h}\|_{L^1(\Omega)}.
\end{equation}   
\end{lemma}
\begin{proof}
Let us pick $\bar u_h=1$ and $\bar u_h=1$ in \eqref{eq:u_h} and \eqref{eq:v_h}, respectively, to obtain
$$
(\partial_t u_h, 1)_h=0
$$   
and
$$
(\partial_t v_h, \bar v_h)_h+(u_h v_h, 1)_h=0,
$$
since $(B_u(u_h) u_h, 1)=0$ by construction in \eqref{def:beta^u_ji}. Thus integrating over $(0,t)$ gives
$$
\int_\Omega u_h(t,\x) {\rm d}\x  = \int_\Omega u_{0h}(\x) {\rm d}\x
$$ 
and
$$
\int_\Omega v_h(t,\x) {\rm d}\x + \int_0^t (u_h(s), v_h(s))_h {\rm d}s = \int_\Omega v_{0h}(\x) {\rm d}\x.
$$ 
From \eqref{Positivity_uh} and \eqref{Positivity_and_DMaxP_vh}, it is easy to conclude that  \eqref{Mass-convervation-uh} and   \eqref{L1-bound-vh} hold. 
\end{proof}
\begin{remark} By the inverse inequality \eqref{inv_global: WlpToWmq} for $\ell=m=1$, $p=\infty$ and $q=1$, we are led to
$$
\max_{t\in[0,T_{\rm max})}\|u_h\|_{L^\infty(\Omega)}\le C_{\rm inv} h^{-2} \|u_h\|_{L^1(\Omega)}=C_{\rm inv} h^{-2} \|u_{0h}\|_{L^1(\Omega)},
$$
from which we conclude that $T_{\rm max}=\infty$. 
\end{remark}

\subsection{\emph{A priori} energy estimates}
We now focus on deriving \emph{a priori} estimates for  $\{u_h, v_h\}_{h>0}$ constructed by \eqref{eq:u_h} and \eqref{eq:v_h}. Briefly our goal is to obtain bounds for $\{v_h\}_{h>0}$ independently of $h$ in $L^\infty(0,\infty; L^2(\Omega))\cap L^2(0,\infty; H^1(\Omega))$. These bounds will result from the inherent positivity of $\{u_h\}_{h>0}$. We will further have a control of the artificial sequence $\{w_h\}_{h>0}$ in $L^\infty_\mathrm{loc}(0,\infty; L^1(\Omega))\cap L^2_{\rm loc}(0,\infty; H^1(\Omega))$, where $w_h= - i_h\log  \frac{v_h}{\|v_{0h}\|_{L^\infty(\Omega)}}$, in order to prove that the precompactness of $\{\log v_h\}_{h>0}$ and the positivity of the limiting function $v$. Afterward we will use the above-mentioned bounds to deal with bounds       
for $\{i_h \log (1+u_h)\}_{h>0}$  in $L^2_\mathrm{loc}(0,\infty; L^1(\Omega))\cap L^2_\mathrm{loc}(0, \infty; H^1(\Omega))$.

We first deal with the \emph{a priori} energy estimates for $\{v_h\}_{h>0}$.
\begin{lemma} Let $(u_h, v_h)$ be the solution of \eqref{eq:u_h}-\eqref{eq:v_h}. Then, for all $t\in(0,\infty)$, we have:
\begin{equation}\label{bound:vh_I}
\|v_h(t)\|^2_{L^2(\Omega)} + \int_0^t \left( \|\nabla v_h\|_{L^2(\Omega)}^2 + \|u_h^{1/2} v_h \|^2_{L^2(\Omega)} \right) \, \ds \le \|v_{0h}\|^2_{L^2(\Omega)}.
\end{equation}
\end{lemma}
\begin{proof} Take $\bar v_h=v_h$ as a test function in \eqref{eq:v_h} to find on noting \eqref{Positivity_uh} that
$$
\frac{1}{2}\frac{{\rm d}}{ {\rm d} t}\|v_h\|^2_h+\|\nabla v_h\|^2_{L^2(\Omega)}+\|u_h^\frac{1}{2} v_h\|^2_h=0.
$$
If this last equality is integrated over $(0,t)$, it follows \eqref{bound:vh_I}. 
\end{proof}

Bound \eqref{bound:vh_I} implies that
\begin{corollary} We have:
\begin{equation}\label{bound:vh_II}
 \{v_h\}_{h >0}\mbox{ is bounded in } L^\infty(0,\infty, L^2(\Omega))\cap L^2_{\rm loc}(0,\infty; H^1(\Omega)).
\end{equation}
\end{corollary}

Next we treat the \emph{a priori} estimates for $\{w_h\}_{h>0}$. A first observation should be that $w_h$ is well-defined due to \eqref{Positivity_and_DMaxP_vh}.    
\begin{lemma}
Let $w_h$ be the auxiliary variable defined by $w_h= - i_h\log  \frac{v_h}{\|v_{0h}\|_{L^\infty(\Omega)}}$. Then, the following estimate holds for all $t\in (0,\infty)$:
\begin{equation}\label{bound:wh_I}
\|w_h(t)\|_{L^1(\Omega)}+\int_0^t \|\nabla w_h(s)\|^2_{L^2(\Omega)} {\rm d}s\le t \|u_{0h}\|_{L^1(\Omega)}+\|w_{0h} \|_{L^1(\Omega)}.
\end{equation}
\end{lemma}
\begin{proof}
Take $\bar v_h=-i_h\frac{1}{v_h}$ in \eqref{eq:v_h} to find
$$
-(\partial_t v_h, i_h\frac{1}{v_h})_h -(\nabla v_h,\nabla i_h\frac{1}{v_h})-(u_h v_h, i_h\frac{1}{v_h})_h=0,
$$
which is in turn manipulated to get
\begin{equation}\label{eq:wh}
\frac{{\rm d}}{{\rm d} t}\|w_h\|_{L^1(\Omega)} +\|\nabla w_h\|^2_{L^2(\Omega)} - \|u_h\|_{L^1(\Omega)}\le0,
\end{equation}
since, from \eqref{ineq:log},   
\begin{equation}\label{sec4.4:lab1}
\begin{array}{rcl}
\displaystyle
-(\nabla v_h,\nabla i_h\frac{1}{v_h})&=&\displaystyle-\sum_{i<j\in I} \delta_{ji}v_h \delta_{ij}\frac{1}{v_h} (\nabla\varphi_{\a_i}, \nabla\varphi_{\a_j})
\\
&=&\displaystyle 
-\sum_{i<j\in  I} \frac{\delta^2_{ji} v_h}{v_j v_i}(\nabla\varphi_{\a_i}, \nabla\varphi_{\a_j})
\\
&\ge&\displaystyle 
-\sum_{i<j\in  I} \delta_{ji}^2\log v_h(\nabla\varphi_{\a_i}, \nabla\varphi_{\a_j})
\\
&=&\displaystyle 
-\sum_{i<j\in  I} \delta^2_{ji}\log\frac{v_h}{\|v_{0h}\|_{L^\infty(\Omega)}} (\nabla\varphi_{\a_i}, \nabla\varphi_{\a_j})
\\
&=&\|\nabla w_h\|^2_{L^2(\Omega)}.
\end{array}
\end{equation}
Next integration over $(0,t)$ leads to \eqref{bound:wh_I}.  
\end{proof}

As a result, we have the following local bounds: 
\begin{corollary}  It follows that 
\begin{equation}\label{bound:wh_II}
 \{w_h\}_{h >0}\mbox{ is bounded in }L^\infty_{\rm loc}(0,\infty; L^1(\Omega))\cap L^2_{\rm loc}(0,\infty; H^1(\Omega)).
\end{equation}
\end{corollary}

The above bound involving $H^1(\Omega)$ is the result of applying Poincaré--Wirtinger's inequality, since, by \eqref{cond_log:v_0},
$$
\frac{1}{|\Omega|}\int_\Omega  w_h(\x,t)\, \dx\le \frac{\vartheta}{|\Omega|}  + t  \frac{\|u_{0h}\|_{L^1(\Omega)}}{|\Omega|}
$$
for all $t\in [0,\infty)$. 

As a corollary from \eqref{eq:wh} and \eqref{sec4.4:lab1}, we obtain the following bound:
\begin{corollary} For any $t\in (0,\infty)$, there holds
\begin{equation}\label{bound:wh_III}
-\int_0^t\sum_{i<j\in  I} \frac{\delta^2_{ji}v_h}{v_j v_i} (\nabla\varphi_{\a_i}, \nabla\varphi_{\a_j})\,\ds\le  t \|u_{0h}\|_{L^1(\Omega)}+\|w_{0h} \|_{L^1(\Omega)}.
\end{equation}
\end{corollary}

Our next step is find \emph{a priori} estimates for $\{i_h\log(1+u_h)\}_{h>0}$.
\begin{lemma} It follows that
\begin{equation}\label{bound:log(1+uh)_I}
\begin{array}{rcl}
\displaystyle
\int_0^t  \|\nabla i_h\log (1+u_h)\|^2_{L^2(\Omega)} {\rm d}s 
&-&\displaystyle
\int_0^t \sum_{i<j\in  I} \frac{\delta^2_{ji} u_h}{(1+u_i)(1+u_j)} (\nabla\varphi_{\a_j}, \nabla\varphi_{\a_j})\,\ds
\\
&+&\displaystyle
\int_0^t \sum_{i<j\in I} \beta_{ji}(u_h, v_h) \frac{\delta_{ji}^2 u_h}{(1+u_i)(1+u_j)} \ds
\\
&\le& 4[(1+t)\|u_{0h}\|_{L^1(\Omega)}+ \|w_{0h} \|_{L^1(\Omega)}]. 
\end{array}
\end{equation}
\end{lemma}
\begin{proof}
First rewrite \eqref{eq:u_h} using $w_h$ as 
$$
(\partial_t u_h, \bar u_h)_h+(\nabla u_h, \nabla \bar u_h)
+(u_h\nabla w_h, \nabla \bar u_h)_*+(B(u_h, v_h) u_h, \bar u_h)=0.
$$
Select $\bar u_h=-i_h\frac{1}{1+u_h}$ in \eqref{eq:u_h} to get
\begin{equation}
\begin{array}{rcl}
\displaystyle
-(\partial_t u_h, i_h \frac{1}{1+u_h})_h&-&\displaystyle (\nabla u_h, \nabla i_h \frac{1}{1+u_h})
\\[1.5ex]
&+&\displaystyle
(u_h\nabla w_h, \nabla i_h \frac{1}{1+u_h})_*
\\[1.5ex] 
&-&\displaystyle
(B(u_h, v_h) u_h,i_h \frac{1}{1+u_h})=0.
\end{array}
\end{equation}

The diffusion term is treated as:
$$
-(\nabla u_h, \nabla i_h \frac{1}{1+u_h})=-\sum_{i<j\in I}\frac{\delta^2_{ji} u_h}{(1+u_j)(1+u_i)} (\nabla\varphi_{\a_j}, \nabla\varphi_{\a_i})\ge 0
$$  
owing to the weak acuteness on $\mathcal{T}_h$.

For the chemotaxis term, we define $ \mathds{I}^*=\{(i,j)\in I\times I : u_j\not = u_i \}$. To estimate it,  we invoke \eqref{New_KS_discretization}, together with \eqref{def:tau_ij} and  \eqref{ineq:log}, as follows: 
\begin{align*}
\displaystyle
(u_h\nabla w_h&,\nabla i_h \frac{1}{1+u_h})_*
\\
=&\displaystyle
\sum_{i<j\in  \mathds{I}^*} \frac{1}{2} \left(\frac{u_i}{u_i+1}+\frac{u_j}{u_j+1}\right)\frac{\delta_{ji}\log(1+u_h)}{\delta_{ij}\frac{1}{1+u_h}} \delta_{ji}w_h \delta_{ij}\frac{1}{1+u_h}(\nabla\varphi_{\a_j}, \nabla\varphi_{\a_i})
\\
\le&\displaystyle
\left(-\sum_{i<j\in  \mathds{I}^*}\delta_{ji}^2w_h (\nabla\varphi_{\a_j}, \nabla\varphi_{\a_i})  \right)^{\frac{1}{2}}
\left(-\sum_{i<j\in  \mathds{I}^*}\frac{\delta_{ji}^2u_h}{(u_i+1)(u_j+1)} (\nabla\varphi_{\a_j}, \nabla\varphi_{\a_i})\right)^{\frac{1}{2}}
\\
\le&\displaystyle
-\frac{1}{2}\sum_{i<j\in  I}\delta_{ji}^2w_h (\nabla\varphi_{\a_j}, \nabla\varphi_{\a_i}) 
-\frac{1}{2}\sum_{i<j\in  I} \frac{\delta_{ji}^2u_h}{(u_i+1)(u_j+1)} (\nabla\varphi_{\a_j}, \nabla\varphi_{\a_i}).
\end{align*}

The stabilizing terns is rewritten as:
$$
\begin{array}{rcl}
\displaystyle
-(B(u_h, v_h) u_h, i_h \frac{1}{1+u_h})&=&\displaystyle-\sum_{i<j\in  I} \beta_{ji}(u_h, v_h) \delta_{ji} u_h\delta_{ji}\frac{1}{1+u_h}
\\
&=&\displaystyle
\sum_{i<j\in I} \beta_{ji} (u_h, v_h) \frac{\delta_{ji}^2u_h}{(u_i+1)(u_j+1)} \ge0.
\end{array}
$$ 
This follows since $\beta_{ji}>0$, for $ i\not=j$, from  \eqref{def:beta^u_ji}.

Plugging all the above computations and integrating over $(0,t)$ yields, on noting \eqref{ineq:log}, that  
$$
\begin{array}{rcl}
\|i_h\log (1+u_{0h})\|_{L^1(\Omega)}&+&\displaystyle\frac{1}{4}\int_0^t  \|\nabla i_h\log (1+u_h)\|^2_{L^2(\Omega)} {\rm d}s 
\\
&-&\displaystyle
\frac{1}{4}\int_0^t \sum_{i<j\in  I} \frac{\delta^2_{ji} u_h}{(1+u_i)(1+u_j)} (\nabla\varphi_{\a_j}, \nabla\varphi_{\a_j})\,\ds
\\
&+&\displaystyle
\int_0^t \sum_{i<j\in  \mathds{I}^*} \beta_{ji}(u_h, v_h) \frac{\delta_{ji}^2 u_h}{(u_i+1)(u_j+1)} \ds
\\
&\le& \displaystyle \|i_h\log (1+u_h)(t)\|_{L^1(\Omega)}+ \frac{1}{2}\int_0^t \|\nabla w_h\|^2_{L^2(\Omega)}{\rm d}s,
\end{array}
$$
\end{proof}
from which follows \eqref{bound:log(1+uh)_I} by \eqref{bound:wh_I} and $\log(1+u_h)\le u_h$. 

\begin{corollary} We have: 
\begin{equation}\label{bound:log(1+uh)_II}
\{i_h\log (1+u_h)\}_{h>0} \mbox{ is bounded in } L^\infty_{\rm loc}(0,\infty; L^1(\Omega))\cap L^2_{\rm loc}(0,\infty; H^1(\Omega)).
\end{equation}
\end{corollary}
\subsection{Weak and strong convergences}
During the process of attaining weak and strong compactness for the sequences generated by system \eqref{eq:u_h}-\eqref{eq:v_h}, a subsequence must be extracted at each occurrence.  However, to maintain clarity in the presentation, this aspect will not be mentioned explicitly. 

\subsubsection{Weak precompactness of $\{v_h\}_{h>0}$, $\{w_h\}_{h>0}$, and $\{\log(1+ u_h)\}_{h>0}$.}
From the bounds  \eqref{bound:wh_II} and \eqref{bound:log(1+uh)_II},  Kakutani's theorem \cite[Th. 3.17]{Brezis} and Banach--Alaoglu--Bourbaki \cite[Th. 3.16]{Brezis}  enable us to establish weak compactcness. Thus there exist functions $v\in L^\infty(0,\infty; L^\infty(\Omega))\cap L^2(0,\infty; L^2(\Omega))$, $w, \Xi_{\log(u+1)} \in  L^2_{\rm loc}(0,\infty; H^1(\Omega)$ such that, as $h\to 0$,
\begin{equation}\label{w*_conv_LinfLinf:vh->v}
v_h \overset{*}{\rightharpoonup} v\quad\mbox{ in }\quad L^\infty(0,\infty; L^\infty(\Omega)),
\end{equation}
\begin{equation}\label{w_conv_L2H1:vh->v}
v_h \rightharpoonup  v\quad\mbox{ in }\quad L^2_{\rm loc}(0,\infty; H^1(\Omega)),
\end{equation}
\begin{equation}\label{w_conv_L2H1:wh->w}
w_h \rightharpoonup w  \quad\mbox{ in }\quad L^2_{\rm loc}(0,\infty; H^1(\Omega)),
\end{equation}
and
\begin{equation}\label{w_conv_L2H1:log(1+uh)->log(1+u)}
 \log(1+u_h) \rightharpoonup  \Xi_{\log(u+1)}\mbox{ in } L^2_{\rm loc}(0,\infty; H^1(\Omega)).
\end{equation}

\subsubsection{Strong precompactness for $\{w_h\}_{h>0}$, $\{\log (1+u_h)\}_{h>0}$, and $\{ u_h \}_{h>0}$.}
The strong convergence of $\{w_h\}_{h>0}$ and $\{\log (1+u_h)\}_{h>0}$ is obtained through Aubin-Lions' theorem. To carry this out, a control of $\{\partial_t\log(1+u_h)\}_{h>0}$ and $\{\partial_t w_h\}_{h>0}$, respectively, in some appropriate spaces are needed. However, the strong convergence of $\{ u_h \}_{h>0}$ in $L^1_{\rm loc}(Q)$ requires using Vitali's theorem. As a result, one can enhance this strong convergence for $\{ w_h \}_{h>0}$ in $L^2_{\rm loc}(0,\infty; H^1(\Omega))$.

$\bullet$ \textbf{Precompactness of $\{v_h\}_{h>0}$.} From letting $\bar v_h=q_h \bar v$ in \eqref{eq:v_h} with $\bar v\in H^1(\Omega)\cap L^\infty(\Omega)$, where $q_h$ was defined in \eqref{def:qh}, there follows 
$$
(\partial_t v_h, q_h \bar v)_h +(\nabla v_h,\nabla q_h \bar v)+(u_h v_h, q_h \bar v)_h=0.
$$
Observe that 
$$
(\partial_t v_h, q_h \bar v)_h=(\partial_t v_h, \bar v).
$$
Moreover, on account of \eqref{Stab_Linf:qh} and \eqref{Stab_H1:qh}, one has   
$$
(\nabla v_h,\nabla q_h \bar v)\le \|\nabla v_h\|_{L^2(\Omega)} \|\nabla q_h \bar v\|_{L^2(\Omega)},
$$
and
$$
(u_h v_h, q_h \bar v)_h\le \|u_h\|_{L^1(\Omega)} \|v_h\|_{L^\infty(\Omega)} \|q_h \bar u\|_{L^\infty(\Omega)}.
$$
Thus one readily sees on noting \eqref{bound:vh_I} that 
$$
\{\partial_t v_h\}_{h>0} \mbox{ is bounded in } L^2_{\rm loc}(0,\infty; (H^1(\Omega)\cap L^\infty(\Omega))')
$$
holds as $h\to 0$ and consequently 
\begin{equation}\label{w_conv_L2(H1Linf)':v_ht->vt}
\partial_t v_h \rightharpoonup  \partial_t v\quad\mbox{ in }\quad L^2_{\rm loc}(0,\infty; (H^1(\Omega)\cap L^\infty(\Omega))'). 
\end{equation}

A straightforward application of Aubin-Lions' lemma leads to 
\begin{equation}\label{s_conv_L2L2:vh->v}
v_h \to v \quad\mbox{ in }\quad L^2_{\rm loc}(0,\infty; L^2(\Omega))\quad\mbox{ as } h\to0. 
\end{equation}

$\bullet$ \textbf{Precompactness of $\{w_h\}_{h>0}$.} Let $\bar v\in H^2(\Omega)$ and  take $\bar u_h = i_h(\frac{q_h \bar v}{v_h} )$ to obtain
$$
-(\partial_t v_h, i_h\frac{q_h \bar v}{v_h})_h -(\nabla v_h,\nabla i_h\frac{q_h \bar v}{v_h})-(u_h v_h, i_h\frac{q_h \bar v}{v_h})_h
=0.
$$
Recalling that $ w_h=- i_h \log\frac{v_h}{\|v_{0h}\|_{L^\infty(\Omega)}}$, it is shown that  
$$
-(\partial_t v_h, i_h\frac{q_h \bar v}{v_h})_h=(\partial_t w_h, \bar v).
$$
For the second term, without loss of generality, one may assume  $v_i\le v_j$, respectively. Thus we have, by \eqref{inv_local:WlpToWmq}, that 
$$
\begin{array}{rcl}
\displaystyle
(\nabla v_h,\nabla i_h\frac{q_h \bar v}{v_h})&=&\displaystyle\sum_{i<j\in I} \delta_{ji} v_h \delta_{ij}\frac{q_h\bar v}{v_h} (\nabla\varphi_{\a_i}, \nabla\varphi_{\a_j})
\\
&=&\displaystyle\sum_{i<j\in I} \delta_{ji}v_h \Big(q_h\bar v_i \delta_{ij}\frac{1}{v_h}+\frac{1}{v_j}\delta_{ji}q_h\bar v\Big) (\nabla\varphi_{\a_i}, \nabla\varphi_{\a_j}) 
\\
&\le&\displaystyle 
- \Big(\| q_h \bar v\|_{L^\infty(\Omega)} + \|\nabla  q_h \bar v\|_{L^2(\Omega)}\Big) \sum_{i<j\in I}  \frac{\delta_{ji}^2 v_h}{v_j v_i}(\nabla\varphi_{\a_i}, \nabla\varphi_{\a_j}).
\end{array}
$$
The third term is estimated  as 
$$
-(u_h v_h, i_h\frac{q_h \bar v}{v_h})_h= (u_h, q_h\bar v_h)_h\le \|u_h\|_{L^1(\Omega)} \|q_h \bar v\|_{L^\infty(\Omega)}.
$$
It is now a simple matter on invoking \eqref{Stab_Linf:qh}-\eqref{Stab_H1:qh} and \eqref{bound:wh_III}  to get 
$$
\{\partial_t w_h\}_{h>0}\mbox{ is bounded in } L^1_{\rm loc}(0,\infty; H^2(\Omega)').
$$
This estimate combined with the bound \eqref{bound:wh_II} gives that $\{w_h\}_{h>0}$ is precompact in \break$L^2_{\rm loc}(0,\infty; L^p(\Omega))$ for $1\le p <\infty$, i.e.,    
\begin{equation}\label{conv:wh->w}
w_h \to w  \quad\mbox{ in } L^2_{\rm loc}(0,\infty; L^p(\Omega))\quad\mbox { as } \quad h\to 0.
\end{equation}
$\bullet$ \textbf{Precompactness of $\{i_h\log(1+u_h)\}_{h>0}$}.
Consider $\bar u\in H^2(\Omega)$ and let $\bar u_h = i_h(\frac{q_h \bar u}{1+u_h} )$ in \eqref{eq:u_h} to obtain 
\begin{equation}\label{bound_d/dt log(1+u):eq}
\begin{array}{rcl}
\displaystyle
(\partial_t u_h, \frac{q_h\bar u}{1+u_h})_h+(\nabla u_h, \nabla  i_h\frac{q_h \bar u}{1+u_h})
+(u_h\nabla w_h, \nabla i_h\frac{ q_h \bar u}{1+u_h} )_*&&
\\
\displaystyle
+(B_u(u_h, v_h) u_h, i_h \frac{q_h \bar u}{1+u_h})&=&0.
\end{array}
\end{equation}
It is plain from \eqref{def:qh}  that 
\begin{equation}\label{bound_d/dt log(1+u):time}
(\partial_t u_h, \frac{q_h\bar u}{1+u_h} )_h=(\partial_t i_h\log(1+u_h), \bar u).
\end{equation}
In view of 
$$
\delta_{ij}\frac{q_h \bar u}{1+u_h}=q_h\bar u_i\delta_{ij}\frac{1}{1+u_h}+\frac{1}{1+u_j}\delta_{ij}q_h\bar u,
$$
there results
$$
\begin{array}{rcl}
\displaystyle
(\nabla u_h, \nabla  i_h \frac{q_h \bar u}{1+u_h} )&=&\displaystyle \sum_{i<j\in  I} \delta_{ji} u_h \delta_{ij}\frac{q_h \bar u}{1+u_h} (\nabla\varphi_{\a_j}, \nabla\varphi_{\a_i})
\\
&=&\displaystyle \sum_{i<j\in  I} \delta_{ji} u_h q_h\bar u_i\delta_{ij}\frac{1}{1+u_h} (\nabla\varphi_{\a_j}, \nabla\varphi_{\a_i})
\\
&&\displaystyle+\sum_{i<j\in  I} \delta_{ji} u_h \frac{1}{1+u_j}\delta_{ij}q_h\bar u (\nabla\varphi_{\a_j}, \nabla\varphi_{\a_i})
\\
&=&\Sigma_{11}+\Sigma_{12}.
\end{array}
$$
It follows that
$$
\Sigma_{11}\le - \|q_h \bar u_h\|_{L^\infty(\Omega)}\sum_{i<j\in  I} \frac{\delta^2_{ji}u_h}{(1+u_i)(1+u_j)} (\nabla\varphi_{\a_j}, \nabla\varphi_{\a_i}).
$$
For $\Sigma_{12}$, it is firstly observed on invoking the mean-value theorem that 
\begin{equation}\label{sec4.5.2:lab1}
\begin{array}{rcl}
\displaystyle
\frac{\delta_{ji} u_h}{1+u_j}-\delta_{ji}\log (1+u_h)&=&\displaystyle\frac{\delta_{ji} u_h}{1+u_j}-\frac{\delta_{ji} u_h}{\theta(1+u_j)+(1-\theta) (1+u_i)}
\\
&=&\displaystyle
\delta_{ji} u_h \left(\frac{1}{1+u_j}- \frac{1}{\theta(1+u_j)+(1-\theta) (1+u_i)}\right)
\\
&=&\displaystyle
(1-\theta)\left(\frac{\delta^2_{ji} u_h}{(1+u_j)(\theta(1+u_j)+(1-\theta) (1+u_i))}\right)
\\
&\le&\displaystyle
\frac{\delta^2_{ji}u_h}{(1+u_i)(1+u_j)}
\end{array}
\end{equation}
and 
\begin{equation}\label{sec4.5.2:lab2}
\delta_{ji}q_h\bar u= h_{ji}\nabla q_h\bar u|_{T_{ji}}\cdot \boldsymbol{r}_{ji},
\end{equation}
where $T_{ji}\in\mathcal{T}_h$ is such that $\a_i,\a_j\in T_{ij}$ and $h_{ji}=|\a_j-\a_i|$. From this, on noting \eqref{inv_global: WlpToWmq}, one obtains
$$
\begin{array}{rcl}
\Sigma_{12}&=&\displaystyle \sum_{i<j\in  I} \left(\frac{\delta_{ji} u_h}{1+u_j}-\delta_{ji}\log (1+u_h)\right)\delta_{ij}q_h\bar u (\nabla\varphi_{\a_j}, \nabla\varphi_{\a_i})
\\
&&\displaystyle
+\sum_{i<j\in  I} \delta_{ji}\log (1+u_h)\delta_{ij}q_h\bar u (\nabla\varphi_{\a_j}, \nabla\varphi_{\a_i})
\\
&\le&
\displaystyle
-\|q_h \bar u_h\|_{L^\infty(\Omega)}\sum_{i<j\in  I} \frac{\delta^2_{ji}u_h}{(1+u_i)(1+u_j)} (\nabla\varphi_{\a_j}, \nabla\varphi_{\a_i})
\\
&&+\|\nabla i_h\log (1+u_h)\|_{L^2(\Omega)} \| \nabla q_h\bar u\|_{L^2(\Omega)}.
\end{array}
$$
Thus
\begin{equation}\label{bound_d/dt log(1+u):diff}
\begin{split}
(\nabla u_h, &\nabla  i_h\frac{q_h \bar u}{1+u_h}) \le \Big(\|q_h\bar u\|_{L^\infty(\Omega)}+\|\nabla q_h \bar u\|_{L^2(\Omega)}\Big) 
\\
&\times\Bigg(- 2\sum_{i<j\in  I} \frac{\delta^2_{ji}u_h}{(1+u_i)(1+u_j)} (\nabla\varphi_{\a_j}, \nabla\varphi_{\a_i})
 +\|\nabla i_h\log(1+u_h)\|_{L^2(\Omega)}\Bigg). 
\end{split}
\end{equation}

Let $u_\#=\min\{u_i, u_j\}$ and $u_{\bar\#}=\max\{u_i, u_j\}$ and recall $ \mathds{I}^*=\{(i,j)\in I : u_j=u_i\}$ with $ \mathds{I}^*_c$ being its complementary. It follows on noting \eqref{New_KS_discretization} that
$$
\begin{array}{rcl}
\displaystyle
(u_h\nabla w_h, \nabla i_h\frac{ q_h \bar u}{1+u_h} )_*&=&\displaystyle\sum_{i<j\in  I}\tau_{ji}(u_h) \delta_{ji} w_h \delta_{ij}\frac{ q_h \bar u}{1+u_h} (\nabla \varphi_{\boldsymbol{a}_j},\nabla\varphi_{\boldsymbol{a}_i})
\\
&=&\displaystyle
\sigma\sum_{i<j\in  I}\tau_{ji}(u_h) \delta_{ji}w_h q_h\bar u_\# \delta_{\#\bar\#}\frac{1}{1+u_h} (\nabla  \varphi_{\boldsymbol{a}_j},\nabla\varphi_{\boldsymbol{a}_i})
\\
&&\displaystyle
+\sigma \sum_{i<j\in I}\tau_{ji}(u_h) \delta_{ji} w_h \frac{1}{1+u_{\bar \#}}\delta_{\#\bar\#}q_h\bar u (\nabla \varphi_{\boldsymbol{a}_j},\nabla\varphi_{\boldsymbol{a}_i})
\\
&:=&\Sigma_{21}+\Sigma_{22},
\end{array}
$$
where $\sigma=1$ if $\#=i$ or $-1$ if $\#=j$. On account of \eqref{def:tau_ij}, a simple calculation shows that 
$$
\begin{array}{rcl}
\Sigma_{21}&=&\displaystyle \sigma \sum_{i<j\in  \mathds{I}^*_c}\Lambda_{ji}(u_h) \delta_{\bar\# \#}\log(1+u_h) \delta_{ji}w_h q_h\bar u_i  (\nabla \varphi_{\boldsymbol{a}_j},\nabla\varphi_{\boldsymbol{a}_i})
\\
&\le&\displaystyle
\|q_h\bar u\|_{L^\infty(\Omega)}\|\nabla i_h\log(1+u_h)\|_{L^2(\Omega)} \|\nabla w_h\|_{L^2(\Omega)}.
\end{array}
$$
Further write
$$
\begin{array}{rcl}
\Sigma_{22}&=&\displaystyle \sigma \sum_{i<j\in  \mathds{I}^*_c} \Lambda_{ij}(u_h)\frac{\delta_{ji}\log(1+u_h)}{\delta_{ij}\frac{1}{1+u_h}} \delta_{ji}w_h\frac{1}{1+u_{\bar \#}}\delta_{\#\bar\#}q_h\bar u  (\nabla \varphi_{\boldsymbol{a}_j},\nabla\varphi_{\boldsymbol{a}_i})
\\
&&\displaystyle
+\sigma\sum_{i<j\in  \mathds{I}^*} u_i \delta_{ji}w_h \frac{1}{1+u_{\bar \#}}\delta_{\#\bar\#}q_h\bar u  (\nabla \varphi_{\boldsymbol{a}_j},\nabla\varphi_{\boldsymbol{a}_i}).
\end{array}
$$
As 
$$
\begin{array}{rcl}
\displaystyle
\frac{\delta_{ji}\log(1+u_h)}{\delta_{ij}\frac{1}{1+u_h}}\frac{1}{1+u_{\bar \#}}&=&\displaystyle(1+u_{\#})\frac{\delta_{ji}\log(1+u_h)}{\delta_{ji}u_h}
\\
&=&\displaystyle
\frac{1+u_{\#}}{\theta(1+u_i)+(1-\theta)(1+u_j)} \le 1,
\end{array}
$$
and
$$
\frac{u_i}{1+u_{\bar \#}}\le \frac{u_i}{1+u_i}\le1,
$$
one has
$$
\Sigma_{22}\le \|\nabla w_h\|_{L^2(\Omega)} \|\nabla q_h \bar u\|_{L^2(\Omega)}.
$$
Therefore,
\begin{equation}\label{bound_d/dt log(1+u): chemo}
\begin{split}
(u_h\nabla  w_h, \nabla i_h\frac{ q_h \bar u}{1+u_h} )_* \le & \Big(\|q_h\bar u\|_{L^\infty(\Omega)} +\|\nabla q_h \bar u\|_{L^2(\Omega)}\Big)
\\
&\times\Big(\|\nabla w_h\|_{L^2(\Omega)}  (1+ \|\nabla i_h\log(1+u_h)\|_{L^2(\Omega)} )\Big)  .
\end{split}
\end{equation}

Finally, we proceed as in bounding \eqref{bound_d/dt log(1+u): chemo}:
$$
\begin{array}{rcl}
\displaystyle
(B(u_h, v_h) u_h, i_h \frac{q_h \bar u}{1+u_h})&=&\displaystyle\sum_{i<j\in  I}\beta_{ji}(u_h, v_h) \delta_{ji} u_h \delta_{ji}\frac{q_h\bar u}{1+u_h}
\\
&=&\displaystyle
\sigma\sum_{i<j\in  I}\beta_{ji}(u_h,v_h) \delta_{ji} u_h q_h\bar u_i \delta_{ji}\frac{1}{1+u_h}
\\
&&\displaystyle
+ \sum_{i<j\in  I}\beta_{ji} (u_h, v_h) \delta_{ji} u_h  \frac{1}{1+u_{i}}\delta_{ji}q_h\bar u
\\
&=&\Sigma_{31}+\Sigma_{32}.
\end{array}
$$
Consequently
$$
\Sigma_{31}\le\|q_h\bar u\|_{L^\infty(\Omega)} \sum_{i<j\in I}\beta_{ji} (u_h, v_h) \frac{\delta^2_{ji} u_h}{(1+u_i)(1+u_j)}
$$
and
$$
\begin{array}{rcl}
\Sigma_{32}&=&\displaystyle\sum_{i<j\in  I}  \beta_{ji}(u_h, v_h)  \left(\frac{\delta_{ji} u_h}{1+u_j}-\delta_{ji}\log (1+u_h)\right)\delta_{ij}q_h\bar u
\\
&&\displaystyle
+ \sum_{i<j\in  I} \beta_{ji}(u_h, v_h) \delta_{ji}\log (1+u_h)\delta_{ij}q_h\bar u.
\end{array}
$$
Assume that 
$$
\beta_{ji}(u_h, v_h)=\alpha_{\a_i}(u_h)(1+\delta_{ji}w_h \Big[\frac{\bar \tau_{ji}(u_h)}{\delta_{ji}u_h}-\frac{u_i}{\delta_{ji} u_h}\Big]\Big)(\nabla \varphi_{\a_j}, \nabla \varphi_{\a_i}),
$$
where
$$
\bar\tau_{ji}(u_h)=\Lambda_{ij}(u_h)\frac{\delta_{ji}\log(1+u_h)}{\delta_{ij}\frac{1}{1+u_h}}.
$$
Thus, we have, by \eqref{sec4.1:B1} and \eqref{sec4.1:B2} together with \eqref{Positivity_uh}, that 
$$
\Big|\frac{\bar \tau_{ji}(u_h)}{\delta_{ji}u_h}-\frac{u_i}{\delta_{ji} u_h}\Big|\le 1,
$$
which implies by Lemma \ref{lm:alpha_i}  that 
$$
|\beta_{ji}(u_h, v_h)|\le - |\delta_{ji} w_h| (\nabla \varphi_{\a_j}, \nabla \varphi_{\a_i}).
$$
Then, from \eqref{sec4.5.2:lab1}, one is led to bounding 
$$
\begin{array}{rcl}
\Sigma_{32}&\le&\displaystyle
 C \|q_h\bar u\|_{L^\infty(\Omega)} \sum_{i<j\in  I}  \beta_{ji}(u_h, v_h)\frac{\delta_{ji}^2 u_h}{(1+u_j)(1+u_i)}
\\
& &\displaystyle
 - \|q_h\bar u\|_{L^\infty(\Omega)} \sum_{i<j\in  I}  |\delta_{ji} w_h|  |\delta_{ji}\log(1+u_h)| (\nabla \varphi_{\a_j}, \nabla \varphi_{\a_i})
\\
&\le&\displaystyle
 C \|q_h\bar u\|_{L^\infty(\Omega)} \sum_{i<j\in  I}  \beta_{ji}(u_h, v_h)\frac{\delta_{ji}^2 u_h}{(1+u_j)(1+u_i)}
\\
&&+ \|q_h\bar u\|_{L^\infty(\Omega)} \|\nabla w_h\|_{L^2(\Omega)} \|\nabla i_h\log(1+u_h)\|_{L^2(\Omega)}.
\end{array}
$$
 Accordingly
\begin{equation}\label{bound_d/dt log(1+u):stab}
\begin{array}{rcl}
\displaystyle
(B(u_h, v_h) u_h, i_h \frac{q_h \bar u}{1+u_h}) &\le&\displaystyle
C \|q_h\bar u\|_{L^\infty(\Omega)} \sum_{i<j\in  I}\beta_{ji}(u_h, v_h) \frac{\delta^2_{ji} u_h}{(1+u_i)(1+u_j)}
\\
&&+ \|q_h\bar u\|_{L^\infty(\Omega)} \|\nabla w_h\|_{L^2(\Omega)} \|\nabla i_h\log(1+u_h)\|_{L^2(\Omega)}.
\end{array}
\end{equation}

Combing  \eqref{bound_d/dt log(1+u):time}, \eqref{bound_d/dt log(1+u):diff}-\eqref{bound_d/dt log(1+u):stab} with \eqref{bound_d/dt log(1+u):eq} yields
$$
\begin{array}{rcl}
(\partial_t i_h\log (1+ u_h), \bar u)&\le& C \Bigg(1+
 \|\nabla i_h\log(1+u_h)\|^2_{L^2(\Omega)}+\|\nabla w_h\|^2_{L^2(\Omega)}
\\
&&\displaystyle
+\sum_{i<j\in I}\Big(-(\nabla\varphi_{\a_j}, \nabla\varphi_{\a_i})+\beta_{ji}(u_h,v_h) \Big)\frac{\delta^2_{ji} u_h}{(1+u_i)(1+u_j)} \Bigg)
\\
&&\qquad\times\Big(\|q_h\bar u\|_{L^\infty(\Omega)}+\|\nabla q_h \bar u\|_{L^2(\Omega)}\Big).
\end{array}
$$   
Integrating over $(0,t)$ and recalling the embedding $H^2(\Omega)$ into  $L^\infty(\Omega)$, we find after using \eqref{bound:wh_I} and \eqref{bound:log(1+uh)_I} that    
$$
\int_0^t \|\partial_t i_h\log(1+u_h)\|_{H^2(\Omega)'}\, \dt\le C [(1+t)\|u_{0h}\|_{L^1(\Omega)}+\|w_{0h}\|_{L^1(\Omega)} + t],
$$   
and therefore 
\begin{equation}\label{bound_L1H2':d/dt_log(1+u_h)}
\{\partial_t i_h\log(1+u_h)\}_{h>0}\mbox{ is bounded in } L^1_{\rm loc}(0,\infty; H^2(\Omega)');
\end{equation}
implying 
\begin{equation}\label{w_conv_L2H2':d/dt_log(1+uh)->d/dt_log(1+u)}
\partial_t\log(1+u_h) \rightharpoonup \partial_t\log(1+u)  \quad\mbox{ in }\quad L^1_{\rm loc}(0,\infty; H^2(\Omega)').
\end{equation}
This latter bound combined with \eqref{bound:log(1+uh)_II} in application of the Aubin-Lions' lemma leads to
\begin{equation}\label{s_conv_L2Lp:log(1+uh)->log(1+u)}
\log(1+u_h) \to \Xi_{\log(u+1)}  \mbox{ in } L^2_{\rm loc}(0,\infty; L^p(\Omega)) \quad\mbox { as } \quad h\to 0, 
\end{equation}
since $H^1(\Omega)$ is compactly embedded into $L^p(\Omega)$ for $1\le p<\infty$.

\subsubsection{Strong precompactness of $\{ u_h\}_{h>0}$.}
To prove that $\{ u_h\}_{h>0}$ is precompact, Vitali's convergence theorem will be used.

$\bullet$ First of all, we want to prove that there exists $u\in L^1(Q)$ such that  
$$
u_h\to u:=e^{ \Xi_{\log(1+u)}} -1\quad \mbox { a.e. in }\quad (0,\infty)\times\Omega \quad\mbox{ as }\quad h\to 0.
$$
In doing this, we know from \eqref{s_conv_L2Lp:log(1+uh)->log(1+u)} that    
$$
i_h\log(1+u_h) \to \Xi_{\log(1+u)}\quad \mbox {a.e. in }\quad (0,\infty)\times\Omega
$$
as $h\to0$. Further, from the fact that $\log (1+x)$  is concave, one finds, on noting $\sum_{i\in I }\varphi_{\a_i}=1$, that  
$$
e^{i_h\log(1+u_h)}-1\le e^{\log(1+u_h)}-1=u_h \mbox{ a. e. in } [0,\infty)\times\Omega.
$$ 
Then, by Lebesgue's dominated convergence theorem, it is deduced, for all $n\in\N$, that    
\begin{equation}\label{conv:e^log(1+uh)->u}
e^{i_h\log(1+u_h)} -1\to u \quad \mbox{ in }\quad L^1((n, n+1)\times\Omega)\quad\mbox{ as }\quad h\to 0.
\end{equation}
Now, for $0<\varepsilon<1$, let $u_\varepsilon=\rho_\varepsilon*u \in L^1(n,n+1; C^0(\bar\Omega)\cap W^{1,\infty}(\Omega)) $ be a spacial regularization of $u$ by convolution, with $\rho_\varepsilon$ being a standard mollifier, such that it satisfies 
\begin{equation}\label{conv:u_eps->u}
u_\varepsilon \to u\quad \mbox{ in } \quad  L^1(Q_n)\quad\mbox{ as }\quad \varepsilon \to 0.
\end{equation}
where $Q_n:=(n,n+1)\times\Omega$.

By Egorov's theorem, it follows that, for each $\eta > 0$, there exists a measurable subset $\mathcal{A}\subset (n,n+1)\times\Omega$ such that $|\mathcal{A}|<\frac{\eta}{2}$, and $\{e^{i_h\log(1+u_h)} -1\}_{h>0}$ converges  uniformly to $u$ on $[(n,n+1)\times\Omega]\backslash \mathcal{A}$. Analogously, there exists $\mathcal{B}\subset(n,n+1)\times\Omega$ such that $|\mathcal{B}|<\frac{\eta}{2}$, and $\{u_\varepsilon\}_{\varepsilon>0}$ such that $u_\varepsilon$ converges to $u$ uniformly on $[(n,n+1)\times\Omega]\backslash \mathcal{B}$. Next write 
$$
u_h- u=i_h e^{i_h\log(1+u_h)} -1 - u= i_h e^{i_h\log(1+u_h)} -1 \pm i_h u_\varepsilon \pm u_\varepsilon- u.
$$
Let $\mathcal{C}=\mathcal{A}\cup \mathcal{B}$. Then, by \eqref{Stab_Linf:ih}, we obtain, on noting \eqref{conv:u_eps->u} and \eqref{conv:e^log(1+uh)->u}, that
$$
\begin{array}{rcl}
\|i_h e^{i_h\log(1+u_h)} -1- i_h u_\varepsilon\|_{L^\infty(Q_n\backslash \mathcal{C})}&\le& \| e^{i_h\log(1+u_h)} -1- u_\varepsilon\|_{L^\infty(Q_n\backslash \mathcal{C})}
\\
&\le&\| e^{i_h\log(1+u_h)} -1- u\|_{L^\infty(Q_n\backslash \mathcal{A})}
\\
&&+\| u- u_\varepsilon\|_{L^\infty(Q_n\backslash \mathcal{B})} \to 0
\end{array}
$$ 
as $(h, \varepsilon)\to (0,0)$, and, by \eqref{err_LinfW1inf:i_h},  
$$
\begin{array}{rcl}
\|i_hu_\varepsilon - u_\varepsilon\|_{L^\infty(Q_n\backslash \mathcal{C})}&\le& C h \|u_\varepsilon\|_{L^\infty(n,n+1; W^{1,\infty}(\Omega))}
\\
&\le& C h \|\rho_\varepsilon\|_{W^{1,\infty}(\R^2)} \|u\|_{L^\infty(n,n+1; L^1(\Omega))} 
\\
&\le& C \dfrac{h}{\varepsilon^3} \|u\|_{L^\infty(n,n+1; L^1(\Omega))} \to 0
\end{array}
$$
as $(h, \varepsilon) \to (0,0)$ if, for instance, $\varepsilon=h^\frac{1}{6}$. Therefore we have proved that, for each $\eta>0$, there exists $\mathcal{C}\subset(0,T)\times\Omega$ such that $|\mathcal{C}|<\eta$, and $\{u_h\}_{h>0}$ converges to $u$ uniformly on $[(n,n+1)\times\Omega]\backslash \mathcal{C}$. As a consequence of this, we have that  
$$
u_h\to u\quad\mbox{ a. e. in }\quad Q_n,
$$
and hence, by Cantor's diagonal argument, that, as $h\to0$,  
\begin{equation}\label{conv_a.e.:uh->u}
u_h\to u\quad\mbox{ a. e. in }\quad (0,\infty)\times\Omega.
\end{equation}

$\bullet$ Next we wish to state that $\{u_h\}_{h>0}$ is equi-integrable, i.e.,
for each $\varepsilon >0$, there exists $\delta_\varepsilon >0$ so that for all $\mathcal{\mathcal{D}}\subset Q_n$ satisfying  $\mu(\mathcal{D})<\delta_\varepsilon$, it implies that  
$$
\sup_{h>0} \int_{\mathcal{D}} |u_h|\, \dx\,\dt < \varepsilon. 
$$

From \eqref{Moser-Trudinger}, we have, on choosing $\bar u_h=2 i_h\log(1+u_h)$, that 
\begin{align*}
\int_\Omega i_h (1+u_h)^2 \,\dx\le&C_\Omega \Big(1+ 2\|\nabla i_h\log(1+u_h)\|^2_{L^2(\Omega)}\Big) 
\\
&\times e^{\displaystyle 2 C_\Omega \Big(\|\nabla i_h\log(1+u_h)\|^2_{L^2(\Omega)}+\|i_h\log(1+u_h)\|_{L^1(\Omega)}\Big)}.
\end{align*}
In view of \eqref{Stab:i_h} and $\log(1+x)\le x$ for $x>0$, one is allowed to bound
$$
\log \frac{1}{|\Omega|}\int_\Omega (1+u_h)^2 \,\dx\le \log \frac{C_\Omega}{|\Omega|}+ 4 C_\Omega\|\nabla i_h\log (1+u_h)\|^2+2C_\Omega \|u_h\|_{L^1(\Omega)}
$$
and hence, after time integration for $t>0$ and use of \eqref{bound:log(1+uh)_I},  
\begin{equation}\label{sec4.5.3:lab1}
\begin{array}{rcl}
\displaystyle
\int_0^t\left(\log \frac{1}{|\Omega|}\int_\Omega (1+u_h)^2 \,\dx \right)\, \dt&\le&\displaystyle \Big( \log\frac{C_\Omega}{|\Omega|}+ 2 C_\Omega \|u_{0h}\|_{L^1(\Omega)}\Big)t  
\\
&&\displaystyle
+4 C_\Omega \int_0^t \|\nabla i_h\log (1+ u_h)\|^2_{L^2(\Omega)}
\\
&\le&\displaystyle \Big( \log\frac{C_\Omega}{|\Omega|}+ 2 C_\Omega \|u_{0h}\|_{L^1(\Omega)}\Big)t  
\\
&&\displaystyle
+16C_\Omega\Big[(1+t)\|u_{0h}\|_{L^1(\Omega)}+\|w_{0h} \|_{L^1(\Omega)} \Big]
\\
&\le& K \Big[1+(1+t)\|u_{0h}\|_{L^1(\Omega)}+\|w_{0h} \|_{L^1(\Omega)} \Big]:=L.
\end{array}
\end{equation}

Fix $n\in\mathds{N}$. Choose $\xi>|\Omega|$ and define 
$$
\mathcal{D}_1(h)=\{t\in(n,n+1): \int_\Omega u^2(t)\, \dx<\xi\}
$$
and
$$
\mathcal{D}_2(h)=\{t\in(n,n+1): \int_\Omega u^2(t)\, \dx\ge\xi\}.
$$
Using \eqref{sec4.5.3:lab1} yields
$$
\begin{array}{rcl}
L&\ge&\displaystyle \int_{\mathcal{D}_2(h)}\left(\log \frac{1}{|\Omega|}\int_\Omega (1+u_h)^2 \,\dx \right)\, \dt 
\\
&\ge&\displaystyle \int_{\mathcal{D}_2(h)}\left(\log \frac{1}{|\Omega|}\int_\Omega u^2_h \,\dx \right)\, \dt
\\
&\ge&\displaystyle |\mathcal{D}_2(h)| \log\frac{\xi}{|\Omega|},  
\end{array}
$$ 
which implies 
$$
|\mathcal{D}_2(h)|\le \frac{L}{ \log\frac{\xi}{|\Omega|}}. 
$$

Let $\mathcal{E} \subset Q_n$ be an arbitrary measurable such that $ |\mathcal{E}| < \delta $ and $ \mathcal{E}(t) = \{ \x \in \Omega \mid (\x,t) \in \mathcal{E}\} $.  Then 
$$
\begin{array}{rcl}
\displaystyle
\int\!\!\!\!\int_\mathcal{E} u_{h}\, \dx\, \dt &=&\displaystyle \int_{\mathcal{D}_1(h)} \int_{\mathcal{E}(t)} u_{h}\, \dx\, \dt+\int_{\mathcal{D}_2(h)} \int_{\mathcal{E}(t)} u_{h}\, \dx\, \dt
\\
&\le&\displaystyle
\int_{\mathcal{D}_1(h)} |\mathcal{E}(t)|^\frac{1}{2} \left(\int_{\Omega}u^2_h \, \dx\right)^{\frac{1}{2}} \, dt + |\mathcal{D}_2(h)| \|u_{0h}\|_{L^1(\Omega)} 
\\
&\le&\displaystyle
\xi^{\frac{1}{2}}  \int_{\mathcal{D}_1(h)} |\mathcal{E}(t)|^\frac{1}{2}\, dt + |\mathcal{D}_2(h)| \|u_{0h}\|_{L^1(\Omega)} 
\\
&\le&\displaystyle
\xi^{\frac{1}{2}} |\mathcal{D}(h)|^{\frac{1}{2}}  \left(\int_{\mathcal{D}_1(h)} |\mathcal{E}(t)|\, \dt \right)^{\frac{1}{2}}  + |\mathcal{D}_2(h)| \|u_{0h}\|_{L^1(\Omega)} 
\\
&\le&\displaystyle
\xi^{\frac{1}{2}} (n+1)^{\frac{1}{2}}  |\mathcal{E}|^\frac{1}{2}  + |\mathcal{D}_2(h)| \|u_{0h}\|_{L^1(\Omega)} 
\\
&\le&\displaystyle 
\xi^{\frac{1}{2}} (n+1)^{\frac{1}{2}} \delta^\frac{1}{2} +   \frac{L}{ \log\frac{\xi}{|\Omega|}}  \|u_{0h}\|_{L^1(\Omega)}.
\end{array}
$$
Next if one chooses  $\xi$ to be sufficiently large such that  
$$
\frac{L}{ \log\frac{\xi}{|\Omega|}}  \|u_{0h}\|_{L^1(\Omega)}\le\frac{\varepsilon}{2} 
$$
and hence $\delta$ to be adequately small so that
$$
\xi^{\frac{1}{2}} (n+1)^{\frac{1}{2}} \delta^\frac{1}{2}\le\frac{\varepsilon}{2},
$$ 
then it follows that
$$
\int\!\!\!\!\int_\mathcal{E} u_{h}\, \dx\, \dt\le\varepsilon.
$$ 
$\bullet$ Now Vitali's theorem ensures that, as $h\to 0$, 
$$
u_h\to u\quad\mbox{ in } L^1(n, n+1; L^1(\Omega))
$$
and Cantor's diagonal argument gives
 \begin{equation}\label{s_conv_L1L1:uh->u}
u_h\to u\quad\mbox{ in } L^1_{\rm loc}(0,\infty; L^1(\Omega)).
\end{equation}
as $h\to 0$.

In order to be able to identify $\Xi_{\log(1+u)}\equiv\log(1+u)$, one needs to simply observe, by \eqref{s_conv_L1L1:uh->u},  that
$$
\log(1+u_h)\to \log(1+u)\quad\mbox{ a. e. in }\quad (0,\infty)\times\Omega,
$$
as $h\to 0$.   
\subsubsection{Strong precompactness of $\{\nabla w_h\}_{h>0}$}
The matter that $\{\nabla w_h\}_{h>0}$ is precompact in $L^2_\mathrm{loc}(0,\infty; L^2(\Omega))$  is significantly more complex. We follow the reasoning presented in \cite[Co. 6.4]{Badia_GS}. Thus our first goal is to pass to the limit as $h\to 0$ in \eqref{eq:v_h}. Let us define $\tilde W^{k,\infty}(\Omega)=H^k(\Omega)\cap W^{k-1,\infty}(\Omega)$, for $k=1,2$, endowed with $\|\cdot\|_{\tilde W^{k,\infty}(\Omega)}=\|\cdot\|_{H^k(\Omega)}+\|\cdot\|_{W^{k-1,\infty}(\Omega)}$,  and $\tilde W^{k,\infty}(\Omega)'$  its dual with $<\cdot, \cdot>_{k,\infty}$ being the dual pairing.  Take $\varphi\in C^\infty_0([0,\infty); \tilde W^{2,\infty}(\Omega))$. Choose $\bar v_h=sz_h \varphi$ in \eqref{eq:v_h} and integrate over $(0,t)$ to have
$$
\int_0^t (\partial_t v_h,  sz_h \varphi)_h\, \dt+\int_0^t (\nabla v_h, \nabla sz_h \varphi)\, \dt+\int_0^t (u_h v_h, sz_h\varphi)_h\,\dt
=0. 
$$
We have, by \eqref{Com_err_H1'W1inf:i_h}, \eqref{err_H1H2:i_h}, and \eqref{err_LinfW1inf:i_h}, that  
\begin{equation}\label{sec4.5.4:lab1}
\begin{array}{rcl}
(\partial_t v_h, sz_h\varphi)_h&\le& \|\partial_t v_h sz_h\varphi-i_h(\partial_t v_h sz_h\varphi)\|_{L^1(\Omega)}+<\partial_t v_h, sz_h\varphi>_{1,\infty} 
\\
&\le& C_{\rm com} h^{\frac{1}{2}} \|\partial_t v_h\|_{\tilde W^{1,\infty}(\Omega)'} \|\nabla sz_h\varphi\|_{L^\infty(\Omega)}
\\
&&+<\partial_t v_h, sz_h\varphi-\varphi >_{1,\infty}+ <\partial_t v_h, \varphi >_{1,\infty} 
\\
&\le& (C_{\rm com} h^{\frac{1}{2}}+ C_{\rm err} h) \|\partial_t v_h\|_{\tilde W^{1,\infty}(\Omega)'} \|\varphi\|_{\tilde W^{2,\infty}(\Omega)}+ <\partial_t v_h, \varphi >_{1,\infty} 
\end{array}
\end{equation}
and therefore, from \eqref{w_conv_L2(H1Linf)':v_ht->vt}, we infer that, as $h\to 0$, 
$$
\int_0^t (\partial_t v_h, sz_h\varphi)_h\, \dt\to\int_0^t <\partial_t v, \varphi>_{1,\infty}\, \dt.
$$ 
Similarly, by \eqref{err_H1H2:i_h},    
$$
\begin{array}{rcl}
(\nabla v_h, \nabla sz_h \varphi)&=&(\nabla v_h, \nabla (sz_h \varphi- \varphi ))+(\nabla v_h, \nabla \varphi)
\\
&\le & C_\mathrm{err} h \|\nabla v_h\|_{L^2(\Omega)} \|\varphi \|_{H^2(\Omega)}+(\nabla v_h, \nabla \varphi);
\end{array}
$$
as a result of \eqref{w_conv_L2H1:vh->v}, as $h\to 0$,
$$
\int_0^t (\nabla v_h, \nabla q_h \varphi )\, \dt\to \int_0^t (\nabla v, \nabla \varphi )\, \dt.
$$
Next, in view of \eqref{Com_err_L2H1:i_h},  \eqref{inv_global: WlpToWmq}, \eqref{Stab_Linf:ih}, and \eqref{err_LinfW1inf:i_h}, we bound as 
$$
\begin{array}{rcl}
(u_h v_h, sz_h\varphi)_h&=& (u_h v_h, sz_h\varphi)_h\pm(i_h(u_h v_h), sz_h \varphi)\pm(u_h v_h, sz_h\varphi)\pm(u_h v_h, \varphi) 
\\
&\le&
\|i_h(i_h(u_h v_h)sz_h\varphi)-i_h(u_h v_h)sz_h\varphi\|_{L^1(\Omega)}
\\
&&+ \|i_h(u_hv_h)-(u_h v_h)\|_{L^1(\Omega)} \|sz_h \varphi\|_{L^\infty(\Omega)}
\\
&&+ \|u_h\|_{L^1(\Omega)} \|v_h\|_{L^\infty(\Omega)} \|\varphi-sz_h\varphi\|_{L^\infty(\Omega)}
\\
&&+(u_h v_h, v)
\\
&\le& C_{\rm com} h^{\frac{1}{2}} \|u_h\|_{L^1(\Omega)}\|v_h\|_{L^\infty(\Omega)} \|\nabla sz_h\varphi\|_{L^2(\Omega)}
\\
&&+ C_{\rm com} h^{\frac{1}{2}} \|u_h\|_{L^1(\Omega)} \|\nabla v_h\|_{L^2(\Omega)} \|\varphi\|_{L^\infty(\Omega)}
\\
&&+ C_{\rm err} h \|u_h\|_{L^1(\Omega)} \|v_h\|_{L^\infty(\Omega)} \|\varphi\|_{W^{1,\infty}(\Omega)} 
\\
&&+(u_h v_h, \varphi)
\\
&\le &C h^{\frac{1}{2}} \|u_h\|_{L^1(\Omega)}(\|v_h\|_{L^\infty(\Omega)} + \|\nabla v_h\|_{L^2(\Omega)}) \|\varphi \|_{W^{1,\infty}(\Omega)}
\\
&&+(u_h v_h, \varphi),
\end{array}
$$
which gives, on using \eqref{Positivity_and_DMaxP_vh}, \eqref{bound:vh_I},  \eqref{Mass-convervation-uh},  \eqref{w*_conv_LinfLinf:vh->v}, and \eqref{s_conv_L1L1:uh->u}, that 
$$
\int_0^t (u_h v_h, sz_h\varphi)_h\,\dt\to \int_0^t (u v, \varphi)\,\dt\quad\mbox{ as }\quad h\to 0.
$$

Consequently, we have found $v\in L^2_\mathrm{loc}(0,\infty; \tilde W^{1,\infty}(\Omega))$ with $\partial_t v\in L^2_\mathrm{loc}(0,\infty; \tilde W^{1,\infty}(\Omega)') $, such that 
$$
\partial_t v - \Delta v + u v=0\quad\mbox{  in } \quad L^2_{\rm loc}(0,\infty; W^{1,\infty}(\Omega)' ),
$$
since $W^{1,\infty}(\Omega)\cap H^2(\Omega)$ is dense in $L^\infty(\Omega)\cap H^1(\Omega)$. 

Next, as a  test function, we take $\bar v =\frac{1}{v+\delta \|v_0\|_{L^\infty(\Omega)}}\in H^1(\Omega)\cap L^\infty(\Omega)$, for $\delta>0$, to get 
$$
<\partial_t v, \bar v>+(\nabla v, \nabla \bar v)+( u v, \bar v)=0,
$$
which can be handled to find 
$$
- \frac{\mathrm{d}}{\dt} \int_\Omega\log(\frac{v+\delta \|v_0\|_{L^\infty(\Omega)}}{(1+\delta) \|v_0\|_{L^\infty(\Omega)}})\, \dx- \int_\Omega \frac{|\nabla v|^2}{(v+\delta \|v_0\|_{L^\infty(\Omega)})^2}\, \dx- \int_\Omega u \frac{v}{v+\delta \|v_0\|_{L^\infty(\Omega)}}=0.
$$
Integrating over $(n,t)$, with $t<n+1$, and thereafter over $(n, n+1)$  for $n\in\mathds{N}$, yields on invoking Fubini's theorem that   
\begin{equation*}
\begin{split}
-\int_n^{n+1} \int_\Omega\log(\frac{v(t)+\delta \|v_0\|_{L^\infty(\Omega)}}{(1+\delta) \|v_0\|_{L^\infty(\Omega)}})\, \dx+ \int_n^{n+1} (n+1-t) \int_\Omega \frac{|\nabla v|^2}{(v+\delta \|v_0\|_{L^\infty(\Omega)})^2}\, \dx \,\dt
\\
- \int_n^{n+1} (n+1-t) \int_\Omega \frac{uv}{v+\delta \|v_0\|_{L^\infty(\Omega)}}\,\dx\,\dt=- n \int_\Omega\log(\frac{v(n)+\delta \|v_0\|_{L^\infty(\Omega)}}{(1+\delta) \|v_0\|_{L^\infty(\Omega)}})\, \dx.
\end{split}
\end{equation*}
From the fact that if $0<\delta_1<\delta_2<1$, it follows that 
$$
-\log(\frac{v(t)+\delta_2 \|v_0\|_{L^\infty(\Omega)}}{(1+\delta_2) \|v_0\|_{L^\infty(\Omega)}})\le -\log(\frac{v(t)+\delta_1 \|v_0\|_{L^\infty(\Omega)}}{(1+\delta_1) \|v_0\|_{L^\infty(\Omega)}})
$$
$$
\frac{|\nabla v|^2}{(v+\delta_2 \|v_0\|_{L^\infty(\Omega)})^2}\le \frac{|\nabla v|^2}{(v+\delta_1 \|v_0\|_{L^\infty(\Omega)})^2}
$$
and
$$
\frac{uv}{v+\delta_2 \|v_0\|_{L^\infty(\Omega)}}\le \frac{uv}{v+\delta_1 \|v_0\|_{L^\infty(\Omega)}}.
$$
Then Beppo Levi's theorem ensures that, as $\delta\to 0$, one gets             
\begin{equation}\label{sec4.5.4:lab2}
\begin{array}{l}
\displaystyle
\int_n^{n+1} \|w\|_{L^1(\Omega)} \, \dt+ \int_n^{n+1} (n+1-t) \|\nabla w\|^2_{L^2(\Omega)}\,\dt
\\
\displaystyle
- \int_n^{n+1} (n+1-t) \|u\|_{L^1(\Omega)}\,\dt= n \|w_0\|_{L^1(\Omega)}.
\end{array}
\end{equation}
Next  we start from \eqref{eq:wh} integrated over $(0, t)$ and then integrated over $(0, s)$ to find  
\begin{equation}\label{sec4.5.4:lab3}
\begin{array}{l}
\displaystyle
\int_n^{n+1} \|w_h\|_{L^1(\Omega)} \, \dt+ \int_n^{n+1} (n+1-t) \|\nabla w_h\|^2_{L^2(\Omega)}\,\dt
\\
\displaystyle
- \int_n^{n+1} (n+1-t) \|u_h\|_{L^1(\Omega)}\,\dt\le n \|w_{0h}\|_{L^1(\Omega)}.
\end{array}
\end{equation}
On comparison  \eqref{sec4.5.4:lab2} with \eqref{sec4.5.4:lab3} and passage to the  limit superior, we deduce 
$$
\limsup_{h\to0}\int_n^{n+1} (n+1-t) \|\nabla w_h\|^2_{L^2(\Omega)}\le \int_n^{n+1} (n+1-t) \|\nabla w\|^2_{L^2(\Omega)}
$$
owing to \eqref{conv:wh->w} and \eqref{s_conv_L1L1:uh->u}. Consequently, 
\begin{equation}\label{s_conv_L2H1:wh->w}
w_h \to w \quad\mbox{ in } L^2_\mathrm{loc}(0,\infty; H^1(\Omega))\quad\mbox{ as }\quad  h\to 0. 
\end{equation}
Furthermore, from \eqref{conv:v0h->v0} and a regularization argument, it leads to  
$$
v_h=\|v_{0h}\|_{L^\infty(\Omega)} i_h e^{-w_h} \to v= \|v_0\|_{L^\infty(\Omega)} e^{-w} \quad\mbox{ in } L^p_\mathrm{loc}(0,\infty; L^p(\Omega))\quad\mbox{ as }\quad  h\to 0
$$
for $p\in[1,\infty)$, which implies that 
$$
v>0 \quad\mbox{ a. e. in }\quad (0,\infty)\times\Omega
$$
and
\begin{equation}\label{w=-log v/v0}
w=-\log \frac{v}{\|v_0\|_{L^\infty(\Omega)}}.
\end{equation}

\subsection{Convergence to a generalized solution}
We are in a position to prove that the sequence of pairs $\{u_h , v_h\}_{h>0}$ converges to a generalized weak solution in the sense of Definition \eqref{def:gener_sol}.  

Let $\psi\in C^\infty_0(\bar\Omega\times[0,\infty))$ be nonnegative and substitute $\bar u_h = i_h(\frac{sz_h \varphi}{1+u_h} )$ into \eqref{eq:u_h} to get after integration over $(0,\infty)$ that   
$$
\begin{array}{rcl}
\displaystyle
\int_0^\infty (\partial_t u_h, \frac{sz_h\varphi}{1+u_h})_h\,\dt+\int_0^\infty(\nabla u_h, \nabla  i_h\frac{sz_h \varphi}{1+u_h})\,\dt &&
\\
\displaystyle
+ \int_0^\infty (u_h\nabla w_h, \nabla i_h\frac{sz_h \varphi}{1+u_h} )_*\, \dt
+\int_0^\infty(B(u_h, v_h) u_h, i_h \frac{sz_h \varphi}{1+u_h})\,\dt&=&0.
\end{array}
$$

$\bullet$ It follows on proceeding as in \eqref{sec4.5.4:lab1} and noting \eqref{bound_L1H2':d/dt_log(1+u_h)}, \eqref{w_conv_L2H2':d/dt_log(1+uh)->d/dt_log(1+u)} and \eqref{conv:u0h->u0}, that, as $h\to 0$,   
$$
\int_0^\infty (\partial_t u_h, \frac{sz_h\varphi}{1+u_h} )_h\,\dt \to -\int_0^\infty (\log(1+u), \partial_t \varphi )\,\dt- (\log(1+u_0), \varphi(0)),
$$
where we have used integration by part.  

$\bullet$ It is evident that 
$$
\begin{array}{rcl}
\displaystyle
\int_0^\infty(\nabla u_h, \nabla  i_h \frac{sz_h \varphi}{1+u_h})\,\dt&=&\displaystyle \int_0^\infty \sum_{i<j\in  I} \delta_{ji} u_h \delta_{ij}\frac{sz_h \varphi}{1+u_h} (\nabla\varphi_{\a_j}, \nabla\varphi_{\a_i})
\\
&=&\displaystyle \int_0^\infty\sum_{i<j\in  I} \delta_{ji} u_h sz_h\varphi_i\delta_{ij}\frac{1}{1+u_h} (\nabla\varphi_{\a_j}, \nabla\varphi_{\a_i})
\\
&&\displaystyle+\int_0^\infty\sum_{i<j\in  I} \delta_{ji} u_h \frac{1}{1+u_j}\delta_{ij}sz_h\varphi (\nabla\varphi_{\a_j}, \nabla\varphi_{\a_i})
\\
&=&\Sigma_{11}+\Sigma_{12}.
\end{array}
$$
For $\Sigma_{11}$, write
$$
\begin{array}{rcl}
\Sigma_{11} &=&\displaystyle
\int_0^\infty\sum_{i<j\in  I} \delta_{ji} u_h (sz_h\varphi_i-\varphi_i)\delta_{ij}\frac{1}{1+u_h} (\nabla\varphi_{\a_j}, \nabla\varphi_{\a_i})
\\
&&\displaystyle
+\int_0^\infty\sum_{i<j\in  I} \delta_{ji} u_h \varphi_i \delta_{ij}\frac{1}{1+u_h} (\nabla\varphi_{\a_j}, \nabla\varphi_{\a_i})
\\
&=&\displaystyle
\int_0^\infty\sum_{i<j\in  I} \delta_{ji} u_h (sz_h\varphi_i-\varphi_i)\delta_{ij}\frac{1}{1+u_h} (\nabla\varphi_{\a_j}, \nabla\varphi_{\a_i})
\\
&&\displaystyle
+\int_0^\infty\sum_{\mbox{\tiny$\begin{array}{c} i<j\in I   \\ k\in I\end{array}$}} \delta_{ji} u_h \varphi_k \delta_{ij}\frac{1}{1+u_h} (\varphi_{\a_k}\nabla\varphi_{\a_j}, \nabla\varphi_{\a_i})
\\
&&\displaystyle
+\int_0^\infty\sum_{\mbox{\tiny$\begin{array}{c} i<j\in I  \\ k\in I\end{array}$}} \delta_{ji} u_h (\varphi_i-\varphi_k) \delta_{ij}\frac{1}{1+u_h} (\varphi_{\a_k}\nabla\varphi_{\a_j}, \nabla\varphi_{\a_i})
\\
&\le&\displaystyle
\int_0^\infty\sum_{i<j\in  I} \delta_{ji} u_h (sz_h\varphi_i-\varphi_i)\delta_{ij}\frac{1}{1+u_h} (\nabla\varphi_{\a_j}, \nabla\varphi_{\a_i})
\\
&&\displaystyle
+\int_0^\infty\sum_{\mbox{\tiny$\begin{array}{c} i<j\in I   \\ k\in I\end{array}$}} \delta_{ji} u_h (\varphi_i-\varphi_k) \delta_{ij}\frac{1}{1+u_h} (\varphi_{\a_k}\nabla\varphi_{\a_j}, \nabla\varphi_{\a_i})
\\
&&
-\|(i_h\varphi^{\frac{1}{2}}-\varphi^\frac{1}{2})\nabla i_h\log(1+u_h)\|_{L^2(\Omega)}^2
\\
&&
-\|\varphi^\frac{1}{2} \nabla i_h\log(1+u_h)\|_{L^2(\Omega)}^2.
\end{array}
$$
Then, a consideration of \eqref{error:sz_h}, \eqref{bound:log(1+uh)_I}, and \eqref{err_LinfW1inf:i_h} yields that the residual terms go to zero as $h\to 0$, and, from the lower semicontinuity of the norm, we infer on noting  \eqref{w_conv_L2H1:log(1+uh)->log(1+u)} that     
$$
\liminf_{h\to 0}  \Sigma_{11}\le -\int_0^\infty \|\varphi^{\frac{1}{2}}\nabla\log(1+u)\|^2_{L^2(\Omega)}\,\dt.
$$
For $\Sigma_{12}$,  note
$$
\begin{array}{rcl}
\Sigma_{12}&=&\displaystyle \int_0^\infty \sum_{i<j\in  I} \left(\frac{\delta_{ji} u_h}{1+u_j}-\delta_{ji}\log (1+u_h)\right)\delta_{ij}sz_h\varphi (\nabla\varphi_{\a_j}, \nabla\varphi_{\a_i})
\\
&&\displaystyle
+\int_0^\infty (\nabla \log (1+u_h), \nabla sz_h\varphi)\,\dt.
\end{array}
$$
In view of \eqref{sec4.5.2:lab1} and \eqref{bound:log(1+uh)_I} and using \eqref{error:sz_h} and the fact $|\delta_{ij} sz_h\varphi|\le C_{\rm sta } h \|\varphi\|_{W^{1,\infty}(\Omega)}$ holds by \eqref{sta:sz_h} yields 
$$
\Sigma_{21}\to \int_0^\infty (\nabla \log (1+u), \nabla \varphi)\,\dt\quad\mbox{ as }\quad h\to0. 
$$
Thus we conclude 
$$
\liminf_{h\to0}\int_0^\infty(\nabla u_h, \nabla  i_h \frac{sz_h \varphi}{1+u_h})\,\dt\le -\int_0^\infty \|\varphi^{\frac{1}{2}}\log(1+u)\|^2_{L^2(\Omega)}\,\dt+\int_0^\infty (\nabla \log (1+u), \nabla \varphi)\,\dt.
$$

$\bullet$ For the chemotaxis term, recall $u_\#=\min\{u_i, u_j\}$ and $u_{\bar\#}=\max\{u_i, u_j\}$.  It follows on noting \eqref{New_KS_discretization} that
$$
\begin{array}{rcl}
\displaystyle
(u_h\nabla w_h, \nabla i_h\frac{ sz_h \varphi}{1+u_h} )_*&=&\displaystyle\sum_{i<j\in  I}\tau_{ji}(u_h) \delta_{ji} w_h \delta_{ij}\frac{ sz_h \varphi}{1+u_h} (\nabla \varphi_{\boldsymbol{a}_j},\nabla\varphi_{\boldsymbol{a}_i})
\\
&=&\displaystyle
\sigma\sum_{i<j\in  I}\tau_{ji}(u_h) \delta_{ji}w_h sz_h\varphi_\# \delta_{\#\bar\#}\frac{1}{1+u_h} (\nabla  \varphi_{\boldsymbol{a}_j},\nabla\varphi_{\boldsymbol{a}_i})
\\
&&\displaystyle
+\sigma \sum_{i<j\in I}\tau_{ji}(u_h) \delta_{ji} w_h \frac{1}{1+u_{\bar \#}}\delta_{\#\bar\#}sz_h\varphi (\nabla \varphi_{\boldsymbol{a}_j},\nabla\varphi_{\boldsymbol{a}_i})
\\
&:=&\Sigma_{21}+\Sigma_{22},
\end{array}
$$
where $\sigma=1$ if $\#=i$ or $-1$ if $\#=j$. Next 
$$
\begin{array}{rcl}
\Sigma_{21}&=&\displaystyle
\sigma \sum_{i<j\in  I} \Lambda_{ij}(u_h)\delta_{ji}\log(1+u_h) \delta_{ji}w_h sz_h\varphi_\#  (\nabla  \varphi_{\boldsymbol{a}_j},\nabla\varphi_{\boldsymbol{a}_i})
\\
&=&\displaystyle
\frac{1}{2}\sum_{\mbox{\tiny$\begin{array}{c}i<j\in I\\ k \end{array}$}} \delta_{ik} \frac{u_h}{1+u_h}  \delta_{ji}\log(1+u_h) \delta_{ji}w_h sz_h\varphi_\#  (\varphi_{\a_k}\nabla  \varphi_{\a_j},\nabla\varphi_{\a_i})
\\
&&+\displaystyle
\frac{1}{2}\sum_{\mbox{\tiny$\begin{array}{c}i<j\in I\\ k \end{array}$}} \delta_{jk} \frac{u_h}{1+u_h}  \delta_{ji}\log(1+u_h) \delta_{ji}w_h sz_h\varphi_\#  (\varphi_{\a_k}\nabla  \varphi_{\a_j},\nabla\varphi_{\a_i})
\\
&&\displaystyle
+\sum_{\mbox{\tiny$\begin{array}{c}i<j\in I\\ k \end{array}$}} \frac{u_k}{1+u_k}\delta_{ji}\log(1+u_h) \delta_{ji}w_h sz_h\varphi_\#  (\varphi_{\a_k}\nabla  \varphi_{\a_j},\nabla\varphi_{\a_i})
\\
&=&\Sigma_{211}+\Sigma_{212}+\Sigma_{213}
\end{array}
$$
and
$$
\begin{array}{rcl}
\Sigma_{22}&=&\displaystyle
\sigma \sum_{i<j\in I}(\tau_{ji}(u_h)-u_i) \delta_{ji} w_h \frac{1}{1+u_{\bar \#}}\delta_{\#\bar\#}sz_h\varphi (\nabla \varphi_{\boldsymbol{a}_j},\nabla\varphi_{\boldsymbol{a}_i})
\\
&&+\displaystyle
\sigma \sum_{i<j\in I}u_i \delta_{ji} w_h \frac{1}{1+u_{\bar \#}}\delta_{\#\bar\#}q_h\varphi (\nabla \varphi_{\boldsymbol{a}_j},\nabla\varphi_{\boldsymbol{a}_i})
\\
&=&\displaystyle
 \frac{\sigma}{2}\sum_{i<j} \left(\frac{u_i}{1+u_i} \frac{\delta_{ji}\log (1+u_h)}{\delta_{ij}\frac{1}{1+u_h}} -u_i\right)\delta_{ji} w_h \frac{1}{1+u_{\bar \#}}\delta_{\#\bar\#}sz_h\varphi (\nabla \varphi_{\boldsymbol{a}_j},\nabla\varphi_{\boldsymbol{a}_i})
 \\
 &&\displaystyle
 +\frac{\sigma}{2}\sum_{i<j} \left(\frac{u_j}{1+u_j} \frac{\delta_{ji}\log (1+u_h)}{\delta_{ij}\frac{1}{1+u_h}} -u_j\right)\delta_{ji} w_h \frac{1}{1+u_{\bar \#}}\delta_{\#\bar\#}sz_h\varphi (\nabla \varphi_{\boldsymbol{a}_j},\nabla\varphi_{\boldsymbol{a}_i})
 \\
 &&\displaystyle
 +\frac{\sigma}{2}\sum_{i<j} (u_j-u_i)\delta_{ji} w_h \frac{1}{1+u_{\bar \#}}\delta_{\#\bar\#}sz_h\varphi (\nabla \varphi_{\boldsymbol{a}_j},\nabla\varphi_{\boldsymbol{a}_i})
 \\
&&+\displaystyle
\sigma \sum_{i<j}u_i \delta_{ji} w_h \frac{1}{1+u_{\bar \#}}\delta_{\#\bar\#}sz_h\varphi (\nabla \varphi_{\boldsymbol{a}_j},\nabla\varphi_{\boldsymbol{a}_i})
\\
&=&\Sigma_{221}+\Sigma_{222}+\Sigma_{223}+\Sigma_{224}.
\end{array}
$$

It will first show that $\int_0^T \Sigma_{211}\,\dt\to 0 $ as $h\to 0$. To carry this out, define $f_h=\frac{u_h}{1+u_h}$ and take $T\in\mathcal{T}_h$ with $\a_i, \a_j$  being two of its vertices. Now observe, by \eqref{ineq:log_II}, that  
$$
|f_i-f_j|\le |\log (1+u_i)-\log (1+u_j)|.
$$  
Thus, by the mean-value theorem,
$$
\begin{array}{c}
|\nabla i_h f_h|_{T}| |\a_i-\a_j| |\cos<\nabla i_h f_h|_{T}, \a_i-\a_j>|
\\
\le |\nabla i_h \log(1+u_h)|_T| |\a_i-\a_j| |\cos<\nabla i_h \log(1+u_h)|_{T}, \a_i-\a_j>. 
\end{array}
$$
As $\frac{x}{1+x}$ and $\log(1+x)$ are strictly monotonically increasing, we have that $\cos<\nabla i_h \log(1+u_h)|_T, \a_i-\a_j>=\cos<\nabla i_h f_h|_T, \a_i-\a_j>$. Therefore, $|\nabla i_h f_h|_{T}||\le |\nabla i_h\log(1+u_h)|_{T}|$ and hence, by \eqref{bound:log(1+uh)_I}, 
\begin{equation}\label{bound_L2H1:ihfh}
\int_0^t \|\nabla i_h f_h\|^2_{L^2(\Omega)}\,\dt\le 4 [  (1+t)\|u_{0h}\|_{L^1(\Omega)}+\|w_{0h} \|_{L^1(\Omega)}],
\end{equation}
for all $t>0$. In view of \eqref{bound:log(1+uh)_I}, \eqref{bound:wh_I}, and  \eqref{bound_L2H1:ihfh}, there exists $N\subset(0,T)$ with $|N|=0$ and $T>0$ being such that $\varphi(t)=0$ for all $t>T$, such that, for each $t\in (0,T)\backslash{N}$,  it follows that one finds $M_t, \widehat M_t,\widetilde M_t>0$ such that, for all $h>0$,  
$$
\|\nabla i_h\log(1+u_h)(t)\|_{L^2(\Omega)}\le M_t,
$$
$$
\|\nabla w_h(t)\|_{L^2(\Omega)}\le \widetilde M_t,
$$
and
$$
\|\nabla i_h f_h(t)\|_{L^2(\Omega)}\le \widehat M_t,
$$
and, furthermore, we have, by \eqref{s_conv_L2H1:wh->w}, that
$$
\|w_h(t)-w(t)\|_{H^1(\Omega)}\to 0\quad \mbox{ as }\quad h\to 0.
$$
From this latter and in application of the absolute continuity of the Lebesgue integral, it follows that, for each $\varepsilon'>0$, there is $\delta>0$ such that $\|\nabla w_h(t)\|_{H^1(\mathcal{A})}<\varepsilon'$ for all $\mathcal{A}\subset\Omega$ with $|\mathcal{A}|<\delta$ and $h>0$. Now let $\lambda>0$. Chebyshev's inequality leads to
$$
|\mathcal{A}_t|:=|\{\x\in\Omega : |\nabla i_h f_h(t) |> \lambda\}|\le \frac{1}{\lambda^2}\|\nabla i_h f_h(t)\|^2_{L^2(\Omega)}\le \frac{1}{\lambda^2} \widehat M_t.
$$
Choosing $\lambda$ to be large enough such that $|\mathcal{A}_t|<\delta$ and, for $\varepsilon''>0$,  $h$ to be small enough such that $\lambda h<\varepsilon''$ yields, on noting \eqref{inv_global: WlpToWmq} and \eqref{sta:sz_h}, that
$$
\begin{array}{rcl}
\Sigma_{211}&\le&\displaystyle C_{\rm inv}\|sz_h \varphi\|_{L^\infty(\Omega)} h \int_{\Omega}|\nabla i_h f_h | |\nabla i_h\log(1+u_h)| |\nabla w_h|\,\dx
\\
&\le&\displaystyle
C_{\rm sta} \|\varphi\|_{L^\infty(\Omega)} h \int_{\Omega\backslash{\mathcal{A}_t}} |\nabla i_h f_h | |\nabla i_h\log(1+u_h)| |\nabla w_h|\,\dx 
\\
&&\displaystyle
+C_{\rm sta} \|\varphi\|_{L^\infty(\Omega)}\int_{\mathcal{A}_t} |\nabla i_h f_h | |\nabla i_h\log(1+u_h)| |\nabla w_h|\,\dx
\\
&\le&\displaystyle
C \|\varphi\|_{L^\infty(\Omega)} h \lambda \|\nabla i_h\log(1+u_h)\|_{L^2(\Omega)}\|\nabla w_h\|_{L^2(\Omega)}
\\
&&\displaystyle
+C_{\rm inv} \|\varphi\|_{L^\infty(\Omega)} \|\nabla i_h\log(1+u_h)\|_{L^2(\Omega)}\|\nabla w_h\|_{L^2(\mathcal{A}_t)}
\\
&\le & C  \|\varphi\|_{L^\infty(\Omega)} \varepsilon'' \|\nabla i_h\log(1+u_h)\|_{L^2(\Omega)}\|\nabla w_h\|_{L^2(\Omega)}
\\
&&+ C  \|\varphi\|_{L^\infty(\Omega)} \varepsilon'  \|\nabla i_h\log(1+u_h)\|_{L^2(\Omega)}.
\end{array}
$$
Let $\varepsilon>0$, and take $\varepsilon'= \frac{\varepsilon}{2C  \|\varphi\|_{L^\infty(\Omega)} M_t}$ and $\varepsilon''= \frac{\varepsilon}{2C  \|\varphi\|_{L^\infty(\Omega)} M_t \widetilde M_t}$ to get 
$$
|\Sigma_{211}|\le \varepsilon,
$$
which implies that, for each $t\in(0,T)\backslash{N}$,  
$$
|\Sigma_{211}(t)|\to 0\quad\mbox{ as }\quad h\to0.
$$
From that the fact that
$$
|\Sigma_{211}(t)|\le \|\nabla i_h\log(1+u_h)(t)\|_{L^2(\Omega)} \|\nabla w_h(t)\|_{L^2(\Omega)}
$$
for all $t\in(0,T)\backslash{N}$, there holds
$$
\int_0^T |\Sigma_{211}|\,\dt\to 0\quad\mbox{ as }\quad h \to 0.             
$$
Analogously,
$$
\int_0^T |\Sigma_{212}|\,\dt\to 0\quad\mbox{ as }\quad h \to 0.             
$$

Let us now see that 
\begin{equation}\label{Sigma_213}
\int_0^T \Sigma_{213}\, \dt \to \int_0^T (\frac{\varphi u}{1+u} \nabla \log(1+u), \nabla \log v)\, \dt.
\end{equation}
Let $T\in\mathcal{T}_h$ with $\a_i, \a_j$ being two of its vertices. Write, by the mean value theorem, 
$$
\begin{array}{rcl}
f_i-f_j&=&|\nabla i_h f_h|_{T}| |\a_i-\a_j| |\cos<\nabla i_h f_h|_{T}, \a_i-\a_j>|
\\
&=& |\nabla f_h|_T(\boldsymbol{\xi_T})| |\a_i-\a_j| |\cos<\nabla f_h |_{T} (\boldsymbol{\xi}_T), \a_i-\a_j>   
\end{array}
$$
with $\boldsymbol{\xi}_T=\alpha \a_i + (1-\alpha) \a_j$ for certain $\alpha\in (0,1)$. As 
$$
\nabla f_h|_T(\boldsymbol{\xi}_T)= \frac{\nabla u_h|_T}{(1+u_h(\boldsymbol{\xi}_T))^2}, 
$$
and 
$$
|\cos<\nabla i_h f_h|_{T}, \a_i-\a_j>|= |\cos<\nabla f_h|_{T}, \a_i-\a_j>,
$$
we infer
\begin{equation}\label{eq_H1:fh_uh}
|\nabla i_h f_h|_T|= \frac{|\nabla u_h|_T|}{(1+u_h(\boldsymbol{\xi}_T))^2}.
\end{equation}
Let $0<\varepsilon'<1$ and take $M>0$ such that $1-\varepsilon'\le\frac{x}{1+x}<1$ for $x>M$. Define
$$
\mathcal{T}^M_h=\{T\in\mathcal{T}_h : \min_{\x\in T} u_h|_T(\x)\ge M\},
$$
and $\mathcal{T}^M_{hc}$ its complementary. Then, for each $t\in(0,T)\backslash{N}$, it follows, on using \eqref{err_LinfW1inf:i_h} and \eqref{eq_H1:fh_uh}, that 
$$
\begin{array}{rcl}
\|i_hf_h(t)- f_h(t)\|_{L^1(\Omega)}&\le& \displaystyle C_{\rm err} h^2 \sum_{T\in\mathcal{T}^M_{hc}}  \|\nabla u_h\|_{W^{1,\infty}(T)} |T|
\\
&&\displaystyle
+\sum_{T\in\mathcal{T}^M_h} \int_T | i_hf_h-f_h|\,\dx
\\
&\le &\displaystyle
C h^2 \sum_{T\in\mathcal{T}^M_{hc}} (1+u_h(\boldsymbol{\xi}_T))^2  \int_T  |\nabla i_h f_h|_T|^2\, \dx 
\\
&&\displaystyle
+ C \varepsilon' \sum_{T\in\mathcal{T}^M_h} \int_T\dx
\\ 
&\le &C (1+M) h^2 \|\nabla i_h f_h(t)\|^2_{L^2(\Omega)} +C \varepsilon' |\Omega|. 
\end{array}
$$          
Integrating over $(0,T)$ gives 
$$
\int_0^T \|i_h f_h - f_h\|_{L^1(\Omega)}\,\dt\le C  (1+M )h^2 \|\nabla i_h f_h\|^2_{L^2(0,T;L^2(\Omega))}+C T \varepsilon' |\Omega|. 
$$
For $\varepsilon>0$, we choose $\varepsilon'=\frac{ \varepsilon}{2C T|\Omega|}$ and $h$ to be small enough such that \break $C(1+M)h^2 \|\nabla i_h f_h\|^2_{L^2(0,T;L^2(\Omega))}<\frac{\varepsilon}{2}$ due to \eqref{bound_L2H1:ihfh}. In this way one can bound
$$
 \int_0^T \|i_h f_h - f_h\|_{L^1(\Omega)}\,\dt<\varepsilon.
$$
and hence
$$
 \int_0^T \|i_h f_h - f_h\|_{L^1(\Omega)}\,\dt \to 0\quad\mbox{ as }\quad h\to0.
$$
As a result,
$$
i_h f_h - f_h \to 0 \quad\mbox{ as }\quad h\to 0\quad\mbox{ a. e. in }\quad \Omega\times(0,T).
$$
Additionally, invoking \eqref{Com_err_L2H1:i_h} and \eqref{sta:sz_h}, one finds, on taking into account \eqref{bound_L2H1:ihfh}, that  
$$
\int_0^T \|i_h( sz_h \varphi i_h f_h)-sz_h\varphi i_h f_h\|_{L^1(\Omega)}\, \dt \to 0  \quad\mbox{ as }\quad h\to 0,
$$
which implies 
$$
i_h(sz_h\varphi i_h f_h)-sz_h\varphi i_h f_h) \to 0 \quad\mbox{ as }\quad h\to 0\quad\mbox{ a. e. in }\quad \Omega\times(0,T).
$$
Next we write
$$
\begin{array}{rcl}
\Sigma_{213}&=&([i_h(sz_h\varphi i_h f_h)- sz_h\varphi i_h f_h]  \nabla i_h\log(1+u_h), \nabla w_h)
\\
&&+(sz_h\varphi (i_h f_h- f_h) \nabla \log(1+u_h), \nabla w_h)
\\
&&+ (sz_h\varphi f_h \nabla \log(1+u_h), \nabla w_h).
\end{array}
$$
Recalling \eqref{s_conv_L1L1:uh->u} and  \eqref{s_conv_L2H1:wh->w}, we infer that 
$$
(i_h( sz_h\varphi i_h f_h)-sz_h\varphi i_h f_h)\nabla w_h\to 0 \quad\mbox{ a.e. in}\quad \Omega\times(0,T)\quad\mbox{ as }\quad h\to 0,
$$

$$
[sz_h\varphi (i_h f_h- f_h)] \nabla w_h\to 0\quad\mbox{ a.e. in}\quad \Omega\times(0,T)\quad\mbox{ as }\quad h\to 0,
$$
and
$$
sz_h\varphi f_h \nabla w_h\to \frac{\varphi u}{1+u} \nabla w\quad\mbox{ a.e. in}\quad \Omega\times(0,T)\quad\mbox{ as }\quad h\to 0,
$$
whereupon, as
$$
|(i_h( sz_h\varphi i_h f_h)-sz_h\varphi i_h f_h)\nabla w_h|\le C_{\rm sta} \|\varphi\|_{L^\infty(\Omega)} |\nabla \tilde w|
$$
$$
|sz_h\varphi(i_hf_h-f_h)] \nabla w_h|\le C_{\rm sta} \|\varphi\|_{L^\infty(\Omega)} |\nabla  \tilde w|
$$
and
$$
|sz_h\varphi f_h \nabla w_h|\le C_{\rm stab} \|\varphi\|_{L^\infty(\Omega)} |\nabla \tilde w|,
$$
it results from \eqref{w=-log v/v0} in
$$
\int_0^T \Sigma_{213}\, \dt\to -\int_0^T (\frac{\varphi u}{1+u} \nabla \log(1+u), \nabla \log v)\, \dt,
$$ 
since 
$$
sz_h\varphi f_h \nabla w_h\to \frac{\varphi u}{1+u} \nabla w\quad\mbox{ in }\quad L^2(0,T; L^2(\Omega))\quad\mbox{ as }\quad h\to 0,
$$
from the Lebesgue dominated convergence theorem. 

From the fact that 
$$
\left|\frac{u_i}{1+u_i} \frac{\delta_{ji}\log (1+u_h)}{\delta_{ij}\frac{1}{1+u_h}} -u_i\right|\le |\delta_{ji} u_h|,
$$
due to \eqref{sec4.1:B1} and \eqref{sec4.1:B2}, we bound
$$
\Sigma_{221}\le C_{\rm sta} h \|\nabla \varphi\|_{L^\infty(\Omega)} \|\nabla w_h\|_{L^2(\Omega)}\left(\sum_{i<j\in I} \frac{\delta_{ji}^2 u_h}{(1+u_i) (1+u_j)}\right)^\frac{1}{2}
$$
and hence 
$$
\int_0^T \Sigma_{221}\, \dt\to 0 \quad\mbox{ as }\quad h\to 0.
$$
Similarly, one gets
$$
\int_0^T \Sigma_{222}\, \dt\to 0 \quad\mbox{ as }\quad h\to 0
$$
and 
$$
\int_0^T \Sigma_{223}\, \dt\to 0 \quad\mbox{ as }\quad h\to 0.
$$
In order to deal with $\Sigma_{224}$, we decompose it as 
$$
\begin{array}{rcl}
\Sigma_{224}&=&\displaystyle
\sigma \sum_{i<j\in I}u_i \delta_{ji} w_h \left(\frac{1}{1+u_{\bar \#}}-\frac{1}{1+u_i}\right)\delta_{\#\bar\#}sz_h\varphi (\nabla \varphi_{\boldsymbol{a}_j},\nabla\varphi_{\boldsymbol{a}_i})
\\
&=&\displaystyle
\sigma \sum_{i<j\in I} \frac{u_i}{1+u_i} \delta_{ji} w_h\delta_{\#\bar\#}sz_h\varphi (\nabla \varphi_{\boldsymbol{a}_j},\nabla\varphi_{\boldsymbol{a}_i})
\\
&\le & \displaystyle
C_{\rm sta} h \|\nabla\varphi\|_{L^\infty(\Omega)} \left(- \sum_{i<j\in I}	\frac{u^2_i}{(1+u_i)^2} \frac{\delta_{ji} u_h}{(1+u_{\bar \#})^2} (\nabla \varphi_{\boldsymbol{a}_j},\nabla\varphi_{\boldsymbol{a}_i})\right)^{\frac{1}{2}} \|\nabla w_h\|_{L^2(\Omega)} 
\\
 &+&\displaystyle
\sum_{i<j\in I} \frac{u_i}{1+u_i} \delta_{ji} w_h\delta_{ij}sz_h\varphi (\nabla \varphi_{\boldsymbol{a}_j},\nabla\varphi_{\boldsymbol{a}_i}).
\end{array}
$$
For the proof of
$$ 
\int_0^T \sum_{i<j\in I} \frac{u_i}{1+u_i} \delta_{ji} w_h\delta_{ij}sz_h\varphi (\nabla \varphi_{\boldsymbol{a}_j},\nabla\varphi_{\boldsymbol{a}_i})\,\dt\to -\int_0^T (\frac{u}{1+u}\nabla\log v, \nabla \varphi)\,\dt\quad{ as }\quad h\to0,
$$
we reason as in \eqref{Sigma_213}.

$\bullet$ Finally, we proceed with the stabilizing term. Write
$$
\begin{array}{rcl}
\displaystyle
(B(u_h, v_h) u_h, i_h \frac{sz_h \varphi}{1+u_h})&=&\displaystyle\sum_{i<j\in I}\beta_{ji}(u_h, v_h)\, \delta_{ji} u_h\, \delta_{ji}\frac{sz_h \varphi}{1+u_h}
\\
&=&\displaystyle
\sum_{i<j\in I}\beta_{ji}(u_h,v_h)\, \delta_{ji} u_h\, sz_h \varphi_i\, \delta_{ij}\frac{1}{1+u_h}
\\
&&\displaystyle
+\sum_{i<j\in I}\beta_{ji}(u_h,v_h) \frac{\delta_{ji} u_h}{1+u_j}\,\delta_{ij}sz_h\varphi
\\
&=&\Sigma_{31}+\Sigma_{32}.
\end{array}
$$

Additionally, for $\Sigma_{31}$,    
$$
\begin{array}{rcl}
\displaystyle
\int_0^T \Sigma_{31} \,\dt&=&\displaystyle
 \int_0^T \sum_{i<j\in I}\beta_{ji}(u_h,v_h) \varphi_i \frac{\delta_{ji}^2 u_h}{(1+u_j)(1+u_i)}\, \dt
\\
&&\displaystyle
+\int_0^T \sum_{i<j\in I}\beta_{ji}(u_h,v_h) (sz_h\varphi_i -\varphi_i) \delta_{ij} \frac{1}{1+u_h}\, \dt.
\end{array}
$$
Observe 
$$
\int_0^T \sum_{i<j\in I}\beta_{ji}(u_h,v_h) \frac{\delta_{ji}^2 u_h}{(1+u_j)(1+u_i)}\, \dt\ge 0.
$$
and, from \eqref{bound:log(1+uh)_I} and  by \eqref{error:sz_h} for $p=\infty$, $m=1$ and $s=0$, 
$$
\int_0^T \sum_{i<j\in I}\beta_{ji}(u_h,v_h) (sz_h\varphi_i -\varphi_i) \frac{\delta_{ij}^2 u_h}{1+u_h}\, \dt \to 0\quad\mbox{ as } h\to 0.
$$
Therefore, 
$$
\liminf_{h\to 0 } \int_0^T \Sigma_{31}(t)\,\dt=
\int_0^T \sum_{i<j\in I}\beta_{ji}(u_h,v_h) \varphi_i \frac{\delta_{ij}^2 u_h}{1+u_h}\, \dt\ge 0.
$$

For $\Sigma_{32}$,  we decompose it as 
$$
\begin{array}{rcl}
\Sigma_{32}&=&\displaystyle
\sum_{i<j\in I}\beta_{ji}(u_h,v_h) \Big( \frac{\delta_{ji} u_h}{1+u_j}-\delta_{ji}i_h\log(1+u_h) \Big) \,\delta_{ij}sz_h\varphi
\\
&&\displaystyle
+\sum_{i<j\in I}\beta_{ji}(u_h,v_h) \delta_{ji}i_h\log(1+u_h)\,\delta_{ij}sz_h\varphi
\\
&:=&\Sigma_{321}+\Sigma_{322}.
\end{array}
$$
On invoking \eqref{sec4.5.2:lab1}  yields
$$
\begin{array}{rcl}
\Sigma_{321}&\le&\displaystyle \sum_{i<j\in I} \beta_{ji}(u_h,v_h) \frac{\delta^2_{ji}u_h}{(1+u_i)(1+u_j)} |\delta_{ij}sz_h\varphi|
\\
&=& \displaystyle
|\delta_{ij}sz_h\varphi|  \sum_{i<j\in I} \beta_{ji}(u_h,v_h)  \frac{\delta^2_{ji}u_h}{(1+u_i)(1+u_j)}
\\
&=& \displaystyle
C_{\rm sta} h \|\nabla\varphi\|_{L^\infty(\Omega)}  \sum_{i<j\in I} \beta_{ji}(u_h,v_h)  \frac{\delta^2_{ji}u_h}{(1+u_i)(1+u_j)}, 
\end{array}
$$
and hence, by \eqref{bound:log(1+uh)_I},  
$$
\int_0^t \Sigma_{321} \, \dt\to0\quad\mbox{ as } \quad h\to 0.  
$$

Noting that 
$$
\begin{aligned} |\beta_{ji}(u_h, v_h)| &\le |\delta_{ji}\log v_h|\, \Big|\frac{u_i}{\delta_{ji} u_h}-\frac{\tau_{ji}(u_h)}{\delta_{ji}u_h}\Big|\, |(\nabla \varphi_{\a_j}, \nabla \varphi_{\a_i})| \\ &\le |\delta_{ji}\log v_h|\, |(\nabla \varphi_{\a_j}, \nabla \varphi_{\a_i})| \qquad \text{(by \eqref{sec4.1:B1}--\eqref{sec4.1:B2} and \eqref{Positivity_uh})} \\ &\le C_{\mathrm{inv}}\,|\delta_{ji}\log v_h| \qquad \text{(by \eqref{inv_global: WlpToWmq} and quasi-uniformity),} \end{aligned}
$$
we obtain 
$$
\begin{array}{rcl}
\Sigma_{322}&\le&\displaystyle C \sum_{i<j\in I} |\delta_{ji}v_h |\delta_{ji}i_h\log(1+u_h) |\delta_{ij}sz_h\varphi|
\\
&=& \displaystyle
 C |\delta_{ij}sz_h\varphi| \left ( \sum_{i<j\in I} |\delta_{ji}v_h|^2  \right )^\frac12 \left ( \sum_{i<j\in I} |\delta_{ji}i_h\log(1+u_h)|   \right )^\frac{1}{2}
\\
&=& C h \|\nabla\varphi\|_{L^\infty(\Omega)} \|\nabla v_h\|_{L^2(\Omega)} \|i_h\log(1+u_h)\|_{L^2(\Omega)}.
\end{array}
$$
Consequently, 
$$
\begin{array}{rcl}
\displaystyle
\int_0^T\Sigma_{322}\, \dt&\le&\displaystyle C_{\rm sta } h \|\varphi\|_{L^\infty(0,T;W^{1,\infty}(\Omega)} \left(\int_0^T \|\nabla v_h\|^2_{L^2(\Omega)}\, \dt \right)^{\frac{1}{2}}
\\
&&\displaystyle
\left(\int_0^T \|\nabla i_h\log(1+u_h)\|^2_{L^2(\Omega)}\, \dt\right)^{\frac{1}{2}}\to 0 \mbox{ as }h\to 0 
\end{array}
$$
from \eqref{bound:wh_I} and \eqref{bound:log(1+uh)_I}.

This completes the proof of Theorem  \ref{Th:main}.

\end{document}